\documentclass[1p]{elsarticle}
\usepackage{amssymb}
\usepackage{amsmath}
\usepackage{amsfonts}
\usepackage{mathtools}
\usepackage{bm}
\usepackage[ruled,vlined]{algorithm2e}
\usepackage{amsthm}
\newtheorem{theo}{Theorem}

\newtheorem{lemm}{Lemma}

\theoremstyle{definition}
\newtheorem{defn}{Definition}
\newtheorem{exmp}{Example}
\newtheorem{asmp}{Assumption}
\newtheorem{rmrk}{Remark}
\newtheorem{fact}{Fact}
\usepackage{setspace}
\usepackage[caption=false,font=footnotesize]{subfig}
\usepackage[draft=false,hidelinks]{hyperref}
\DeclareMathOperator*{\minimize}{minimize}
\DeclareMathOperator*{\argmin}{argmin}

\begin{document}

\begin{frontmatter}

\title{A Convex-Nonconvex Framework for Enhancing\\Minimization Induced Penalties}

\author[1]{Hiroki Kuroda}
\ead{kuroda@vos.nagaokaut.ac.jp}
\affiliation[1]{organization={Department of Information and Management Systems Engineering, Nagaoka University of Technology},
            addressline={1603-1 Kamitomioka}, 
            city={Nagaoka},
            postcode={940-2188}, 
            state={Niigata},
            country={Japan}}

\begin{abstract}
This paper presents a novel framework for nonconvex enhancement of minimization induced (MI) penalties while preserving the overall convexity of associated regularization models. MI penalties enable the adaptation to certain signal structures via minimization, but often underestimate significant components owing to convexity. To overcome this shortcoming, we design a generalized Moreau enhanced minimization induced (GME-MI) penalty by subtracting from the MI penalty its generalized Moreau envelope. While the proposed GME-MI penalty is nonconvex in general, we derive an overall convexity condition for the GME-MI regularized least-squares model. Moreover, we present a proximal splitting algorithm with guaranteed convergence to a globally optimal solution of the GME-MI model under the overall convexity condition. Numerical examples illustrate the effectiveness of the proposed framework.
\end{abstract}


\end{frontmatter}
\biboptions{sort&compress}

\section{Introduction}
\label{sect:Introduction}

Regularization is a key technique for general linear inverse problems in signal processing, machine learning, and data science.
In particular, $\ell_1$ regularization is widely used for sparse estimation \cite{Tibshirani:LASSO,Chen:BP,Baraniuk:CS,Candes:CS,Bruckstein:Sparse}.
While computationally tractable,
$\ell_1$ regularization often underestimates large-magnitude components
since the value of the $\ell_1$ penalty increases with magnitude, owing to convexity.
To mitigate this underestimation tendency,
nonconvex penalties, e.g., \cite{Fan:SCAD,Zhang:MCP} among many others,
have been proposed as better approximations of the $\ell_0$ pseudo-norm
that counts the number of nonzero components.
However, nonconvex penalties usually introduce undesirable local minima,
and their over-reliance on initial values of optimization algorithms
is problematic.

Recently, to resolve the aforementioned trade-off, the so-called \emph{convex-nonconvex strategy} \cite{Blake:CPNC,Nikolova:CPNC,Selesnick:MSCO,Mollenhoff:CPNC,Bayram:CPNC,Malek:CPNC,Selesnick:GMC,Yin:CPNC_PCP,Lanza:GMC,Abe:LiGME,AlShabili:SRS,Yukawa:LiMES,Kitahara:LiGME_MRI,Yata:cLiGME},
which uses a nonconvex penalty that preserves convexity of the overall regularization model,
has garnered significant attention.
After the pioneering works of \cite{Blake:CPNC,Nikolova:CPNC} and
recent case studies by \cite{Selesnick:MSCO,Mollenhoff:CPNC,Bayram:CPNC,Malek:CPNC,Selesnick:GMC,Yin:CPNC_PCP},
several general frameworks \cite{Lanza:GMC,Abe:LiGME,AlShabili:SRS,Yukawa:LiMES} have been developed to
remedy the underestimation tendencies of convex \emph{prox-friendly} penalties (i.e., whose proximity operators can be computed with low complexity), including the $\ell_1$ penalty.
More precisely, these frameworks construct nonconvex penalties by subtracting the \emph{generalized Moreau envelopes} of the convex prox-friendly penalties from them,
and provide proximal splitting algorithms to compute optimal solutions of
the associated regularized least-squares models under overall convexity conditions.
Extensions to handle additional convex constraints have also been studied in \cite{Kitahara:LiGME_MRI,Yata:cLiGME}.

As another important direction for improving regularization,
advanced convex penalties have been designed via minimization to incorporate certain signal structures flexibly.
We refer to this class of penalties as \emph{minimization induced (MI)} penalties.
For example, 
to resolve the problem of unknown block partition in block-sparse estimation, which is encountered
in various applications
\cite{Yu:Audio,Cevher:Video,Gao:Video,Hur:Communications,Kuroda_FPC,Kitahara:PAWR_GRSS,Kuroda:PSDEstimation},
the latent optimally partitioned (LOP)-$\ell_2/\ell_1$ penalty \cite{Kuroda:BlockSparse}
has been designed to minimize
the mixed $\ell_2/\ell_1$ norm \cite{Yuan:l12} over the set of candidate block partitions.
Another prominent example is
the total generalized variation (TGV) penalty \cite{Bredies:TGV},
which is defined by minimization for computing and penalizing
the magnitudes of discontinuous jumps of high-order derivatives.
The TGV penalty is a sound high-order extension of the total variation penalty \cite{Rudin:TV}
and is widely used in image processing applications
\cite{Knoll:TGV,Valkonen:TGV,Ferstl:TGV,Ono:VTGV,Langkammer:TGV,Bredies:TGV_Manifold,Bredies:TGV_Review}.
MI penalties are capable of automatically adapting to certain signal structures;
however, analogous to the $\ell_1$ penalty, they often underestimate the magnitudes of significant components.
In particular, the LOP-$\ell_2/\ell_1$ penalty and the TGV penalty
tend to underestimate blocks of large-magnitude components
and large jumps of derivatives, respectively.

Unfortunately, existing convex-nonconvex frameworks
cannot be directly used to remedy the underestimation tendencies of MI penalties.
Even the most general frameworks \cite{Abe:LiGME,Yukawa:LiMES,Yata:cLiGME}
rely on the assumption that
the convex penalty to be enhanced
can be decomposed into a sum of linearly-involved prox-friendly functions.
This assumption does not hold for MI penalties
because the involved minimization makes their proximity operators difficult to compute.
Thus, a major question arises: Can we remedy the underestimation tendencies of MI penalties
while keeping computation of optimal solutions of associated regularization models tractable?

To address this question,
this paper presents a novel convex-nonconvex framework to enhance MI penalties.
The main contributions of this paper are summarized as follows.
\begin{enumerate}
\item After introducing a class of MI penalties defined by minimization of \emph{seed functions}, we construct 
\emph{generalized Moreau enhanced
minimization induced (GME-MI)} penalties by subtracting from MI penalties their generalized Moreau envelopes.
This procedure yields novel nonconvex enhancements of the LOP-$\ell_2/\ell_1$ penalty and the TGV penalty.
\item By applying the proposed GME-MI penalty to regularized least-squares estimation
for a general linear inverse problem,
we design the GME-MI model and present its overall convexity condition, where
a mild technical assumption is imposed on the seed function.
We prove that this assumption automatically holds for the seed functions used to define the LOP-$\ell_2/\ell_1$ penalty and the TGV penalty.
\item We present a proximal splitting algorithm which is guaranteed to converge to a globally optimal solution of the proposed GME-MI model under the overall convexity condition.
We derive the proposed algorithm by 
carefully designing an averaged nonexpansive operator that
characterizes the solution set of the GME-MI model.
The proposed algorithm is composed only of simple operations:
matrix-vector multiplication, the proximity operators of component functions of the seed function,
and the projection onto an additional convex constraint set.
All the assumed conditions are shown to be valid in the scenarios for enhancement of the LOP-$\ell_2/\ell_1$ penalty and the TGV penalty.
\end{enumerate}

\subsection{Comparison With Related Work}
Besides the aforementioned models,
various nonconvexly regularized models
and their optimization algorithms have been proposed in many studies, e.g., \cite{Candes:ReweightL1,Gasso:DC,Ge:Schatten,Chouzenoux:MM,Soubies:CEL0,Woodworth:NC,Quan:NC,Zhang:NC}.
However, even if the nonconvex penalties are combined with convex data-fidelity functions, these studies do not attempt to maintain convexity of the overall regularization models.
As a result, theoretical results shown in, e.g., \cite{Gasso:DC,Ge:Schatten,Chouzenoux:MM,Soubies:CEL0,Woodworth:NC,Quan:NC,Zhang:NC},
without the overall convexity, only guarantee convergence to certain critical points, which might be undesirable local minima, of their models. In contrast, this paper establishes the convergence of the proposed algorithm to a global minimizer of the newly designed GME-MI model
under the overall convexity condition.

Next, we clarify the novelty of this paper compared to the existing convex-nonconvex frameworks
that guarantee the overall convexity of their models.
While the pioneering studies \cite{Blake:CPNC,Nikolova:CPNC} as well as 
recent studies \cite{Selesnick:MSCO,Chen:CPNC,Mollenhoff:CPNC,Parekh:CPNC_Nuclear,Bayram:CPNC,Malek:CPNC,Lanza:CPNC_MM,Selesnick:CPNC_TV,Suzuki:GME} rely on the the existence of strongly convex term,
the seminal study \cite{Selesnick:GMC} presents a systematic way to
maintain convexity of the model defined as the sum of the GME-$\ell_1$ penalty and a general quadratic data-fidelity function.
This approach has been extended in \cite{Yin:CPNC_PCP} to improve stable principal component pursuit.
After these case studies,
the general frameworks \cite{Lanza:GMC,Abe:LiGME,AlShabili:SRS,Yukawa:LiMES}
have established the systematic methods to design their GME-type penalties from convex prox-friendly penalties,
while guaranteeing the overall convexity of their GME-type regularized least-squares models.
Since these studies do not consider enhancement of MI penalties,
this paper designs the new
systematic construction of GME-MI penalties\footnote{GME-MI penalties are substantially different from the (implicit) penalties, e.g., \cite{Cohen:REDPRO,Belkouchi:LearnPen} and references therein, induced by supervised learning, since GME-MI penalties do not need any dataset of original signals.} as nonconvex enhancement of MI penalties,
and presents the overall convexity condition for the GME-MI regularized least-squares model (Theorem \ref{theo:ExpressCostFunc_SumOfConvex}).

The proposed optimization algorithm for the GME-MI model is also novel to the literature. 
It should be noted that, even if the overall convexity is guaranteed,
it remains a major challenge to establish an algorithm with guaranteed convergence to a global minimizer
of the GME-MI model due to its nonsmoothness and involved nonconvex term.
More precisely, while proximal splitting algorithms have become very popular for nonsmooth convex optimization problems,
the standard algorithms (see, e.g., \cite{Combettes:ProxSplit,Condat:ProxSplit}) are not applicable to the GME-MI model since
they are designed to minimize the sum of convex functions.
It is also challenging to extend the special proximal splitting algorithms
\cite{Selesnick:GMC,Yin:CPNC_PCP,Lanza:GMC,Abe:LiGME,Kitahara:LiGME_MRI,AlShabili:SRS,Yata:cLiGME,Yukawa:LiMES}
for their GME models to the GME-MI model
because the existing frameworks heavily rely on the assumption that the convex penalty to be enhanced is prox-friendly. MI penalties are usually defined via difficult convex minimization (see Example \ref{exmp:MI_penalties}),
and thus their proximity operators are difficult to compute.
We address this challenge by nontrivial reformulation of the solution set of the GME-MI model
to the fixed point set of the proposed averaged nonexpansive operator (Theorem \ref{theo:Convergence_OptimizationAlg}).
This reformulation enables us to establish the convergence of the proposed algorithm
based on the fixed point iteration of the proposed operator
to a global minimizer of the GME-MI model.
The proposed operator is carefully designed so that it requires only the proximity operators of component functions of the seed function
(that is used to define the MI penalty),
which can be computed with low computational complexity, e.g.,
in the scenarios of enhancement of the LOP-$\ell_2/\ell_1$ and TGV penalties (see Example \ref{exmp:representation_MIpenalty_forOptimization} and \ref{appendix:ComputeProxOp}).

\subsection{Paper Organization}
The rest of this paper is organized as follows.
Mathematical preliminaries are presented in Section \ref{sect:Preliminaries}.
In Section \ref{sect:design_GME_MI_penalty_model},
we design the GME-MI model and
derive its overall convexity condition.
In Section \ref{sect:OptimizationAlgorithm}, we present the proposed algorithm for computing a globally optimal solution of the GME-MI model.
Numerical examples for the enhancement scenarios of the LOP-$\ell_2/\ell_1$ and TGV penalties
are presented in Section \ref{sect:Experiment}.
Finally, the paper is concluded in Section \ref{sect:Conclusion}.

A preliminary short version of this paper was presented at a conference \cite{Kuroda:GMEMI_conf}.

\section{Preliminaries}
\label{sect:Preliminaries}
\subsection{Notations}
Let $\mathbb{N}$, $\mathbb{R}$, $\mathbb{R}_{+}$, and $\mathbb{R}_{++}$
denote the sets of nonnegative integers,
real numbers, nonnegative real numbers, and positive real numbers, respectively.
For $p \geq 1$,
the $\ell_p$ norm of $\bm{x} \in \mathbb{R}^{n}$ is defined by
$\|\bm{x}\|_p \coloneqq (\sum_{i=1}^{n}|x_i|^{p})^{1/p}$.
The cardinality of a set
$S$ is denoted by $|S|$.
For $\bm{x}\in\mathbb{R}^n$ and $\mathcal{I} \subset \{1,\ldots,n\}$,
the subvector of 
$\bm{x}$ indexed by $\mathcal{I}$ is denoted by
$\bm{x}_{\mathcal{I}} \coloneqq  (x_i)_{i \in \mathcal{I}} \in \mathbb{R}^{|\mathcal{I}|}$.
Let $\mathcal{G} = (\mathcal{I}_k)_{k=1}^{m}$ be a partition of $\{1,\ldots,n\}$, i.e.,
$\bigcup_{k=1}^{m}\mathcal{I}_k = \{1,\ldots,n\}$,
$\mathcal{I}_{k} \neq \varnothing \quad (k = 1,\ldots,m)$,
and $\mathcal{I}_{k_{1}} \cap \mathcal{I}_{k_{2}} = \varnothing \quad (k_{1} \neq k_{2})$.
Then,
we define the mixed $\ell_2/\ell_1$ norm of $\bm{x} \in \mathbb{R}^{n}$ with 
$\mathcal{G}$
by $\|\bm{x}\|_{2,1}^{\mathcal{G}} \coloneqq  \sum_{k=1}^{m}\sqrt{|\mathcal{I}_k|}\|\bm{x}_{\mathcal{I}_k}\|_2$.
For $\bm{L}\in\mathbb{R}^{m \times n}$, its transpose is denoted by $\bm{L}^{\top}$.
For $\bm{L}\in\mathbb{R}^{m \times n}$ and $S \subset \mathbb{R}^{n}$, 
$\bm{L}(S) \coloneqq  \{\bm{L}\bm{x} \mid \bm{x} \in S\}$.
The range space of $\bm{L}\in\mathbb{R}^{m \times n}$ is denoted by
$\mathrm{ran}(\bm{L})\coloneqq \bm{L}(\mathbb{R}^{n})$.
We denote the identity matrix by $\bm{I}$
and the zero matrix by $\bm{O}$.

Throughout this paper, a Euclidean space $\mathbb{R}^n$
is equipped with the standard inner product $\langle\cdot,\cdot \rangle$ and the associated norm $\|\cdot\|$. For $S \subset \mathbb{R}^{n}$,
$\mathrm{ri}(S)$
denotes the relative interior of $S$ (see, e.g., \cite[Fact 6.14(i)]{BC:ConvexAnalysis}).
We define the operator norm of $\bm{L} \in \mathbb{R}^{m \times n}$ by
$\|\bm{L}\|_{\mathrm{op}} \coloneqq  \max \{\|\bm{L}\bm{x}\| \mid \bm{x} \in \mathbb{R}^{n}, \|\bm{x}\| \leq 1\}$.
The positive definiteness and positive semidefiniteness of
a symmetric matrix $\bm{L}\in\mathbb{R}^{n \times n}$
are expressed by $\bm{L} \succ \bm{O}$ and $\bm{L} \succeq \bm{O}$, respectively.
Unless otherwise stated, we regard the product space
$\mathcal{X} \coloneqq  \bigtimes_{k=1}^{m}\mathbb{R}^{n_k}$
as a Euclidean space with the standard inner product and the associated norm.
In part of this paper, with a positive definite matrix $\bm{P}\in\mathbb{R}^{\mathrm{dim}\mathcal{X}\times\mathrm{dim}\mathcal{X}}$,
we use another inner product $\langle\bm{x},\bm{y} \rangle_{\bm{P}} \coloneqq  \langle\bm{x},\bm{P}\bm{y} \rangle$
for $\bm{x}, \bm{y} \in \mathcal{X}$,
and denote its associated norm by $\|\cdot \|_{\bm{P}}$.

\subsection{Tools in Convex Analysis and Optimization}
A function $f\colon\mathbb{R}^{n}\rightarrow \mathbb{R} \cup \{\infty\}$
is said to be proper if its effective domain
$\mathrm{dom}(f) \coloneqq  \{\bm{x} \in \mathbb{R}^{n} \mid f(\bm{x}) < \infty \}$ is nonempty,
lower semicontinuous if its lower level set
$\{\bm{x} \in \mathbb{R}^{n} \mid f(\bm{x}) \leq a \}$ is closed for every $a \in \mathbb{R}$,
and convex if
$f(\alpha \bm{x} + (1-\alpha)\bm{y}) \leq \alpha f(\bm{x}) + (1-\alpha)f(\bm{y})$
for every $\bm{x},\bm{y} \in \mathbb{R}^{n}$ and $\alpha \in (0,1)$.
The set of all proper lower semicontinuous convex functions
from $\mathbb{R}^{n}$ to $\mathbb{R}\cup\{\infty\}$ is denoted by $\Gamma_0(\mathbb{R}^{n})$.
A function $f\colon\mathbb{R}^{n}\rightarrow \mathbb{R}\cup\{\infty\}$
is said to be coercive if $\lim_{\|\bm{x}\|\rightarrow\infty} f(\bm{x}) = \infty$.

The subdifferential of $f \in \Gamma_0(\mathbb{R}^{n})$ at $\bm{x}\in\mathbb{R}^{n}$ is defined by
\begin{align*}
\partial f(\bm{x}) \coloneqq  \{\bm{u} \in \mathbb{R}^{n} \mid (\forall \bm{y} \in \mathbb{R}^{n})\quad \langle\bm{y}-\bm{x},\bm{u} \rangle + f(\bm{x}) \leq f(\bm{y}) \}.
\end{align*}
We have the relation $\bm{x}^{\star} \in  \argmin_{\bm{x}\in\mathbb{R}^{n}} f(\bm{x}) \Leftrightarrow \bm{0} \in \partial f(\bm{x}^{\star})$.
If $f$ is differentiable at $\bm{x}\in\mathbb{R}^{n}$ with the gradient $\nabla f(\bm{x}) \in \mathbb{R}^{n}$, then $\partial f(\bm{x}) = \{\nabla f(\bm{x})\}$.

\begin{fact}[Subdifferential calculus {\cite[Theorem 16.47]{BC:ConvexAnalysis}}]\hfill
\label{fact:SubdifferentialCalculus}
\begin{enumerate}
\item[a)] (Sum and chain rule).
Let $f \in \Gamma_0(\mathbb{R}^{n})$, $g \in \Gamma_0(\mathbb{R}^{m})$, and $\bm{L} \in \mathbb{R}^{m \times n}$.
Assume $\bm{0} \in \mathrm{ri}[\mathrm{dom}(g) - \bm{L}(\mathrm{dom}(f))]$.
Then, for every $\bm{x} \in \mathbb{R}^{n}$,
\begin{align*}
\partial (f + g\circ \bm{L})(\bm{x}) = \partial f(\bm{x}) + \bm{L}^{\top}(\partial g(\bm{L}\bm{x})).
\end{align*}
\item[b)] (Sum rule).
Let $f \in \Gamma_0(\mathbb{R}^{n})$ and $g \in \Gamma_0(\mathbb{R}^{n})$.
Assume $\mathrm{dom}(g) = \mathbb{R}^{n}$.
Then, for every $\bm{x} \in \mathbb{R}^{n}$,
\begin{align*}
\partial (f+g)(\bm{x}) = \partial f(\bm{x}) + \partial g(\bm{x}).
\end{align*}
\item[c)] (Chain rule).
Let $g \in \Gamma_0(\mathbb{R}^{m})$ and $\bm{L} \in \mathbb{R}^{m \times n}$.
Assume $\bm{0} \in \mathrm{ri}\left[\mathrm{dom}(g) - \mathrm{ran}(\bm{L})\right]$.
Then, for every $\bm{x} \in \mathbb{R}^{n}$,
\begin{align*}
\partial (g\circ \bm{L})(\bm{x}) = \bm{L}^{\top}(\partial g(\bm{L}\bm{x})).
\end{align*}
\end{enumerate}
\end{fact}
\noindent
The conjugate of $f \in \Gamma_0(\mathbb{R}^{n})$ is defined by
\begin{align}
\label{eq:def:conjugate_convexFunc}
f^{*}\colon\mathbb{R}^{n}\rightarrow\mathbb{R}\cup\{\infty\}\colon\bm{u} \mapsto \sup_{\bm{x} \in \mathbb{R}^{n}}\left[\langle \bm{x},\bm{u} \rangle-f(\bm{x})\right],
\end{align}
and $f^{*} \in \Gamma_0(\mathbb{R}^{n})$ \cite[Corollary  13.38]{BC:ConvexAnalysis}.
For any $f \in \Gamma_0(\mathbb{R}^{n})$,
$(\partial f)^{-1} = \partial f^{*}$ \cite[Theorem 16.29]{BC:ConvexAnalysis}, i.e.,
\begin{align}
\label{eq:relation_subdifferential_conjugate}
\bm{u} \in \partial f(\bm{x}) \Leftrightarrow \bm{x} \in \partial f^{*}(\bm{u}).
\end{align}
The proximity operator of $f \in \Gamma_0(\mathbb{R}^{n})$ is defined by
\begin{align*}
\mathrm{prox}_{f}\colon\mathbb{R}^{n}\rightarrow\mathbb{R}^{n}\colon\bm{x} \mapsto \argmin_{\bm{y} \in \mathbb{R}^{n}}
\left[f(\bm{y}) + \frac{1}{2}\|\bm{x}-\bm{y}\|^2\right].
\end{align*}
Let $\mathrm{Id}$ denote the identity operator.
For any $f \in \Gamma_0(\mathbb{R}^{n})$, $\mathrm{prox}_{f} = (\mathrm{Id}+\partial f)^{-1}$
\cite[Proposition 16.44]{BC:ConvexAnalysis}, i.e.,
\begin{align}
\label{eq:relationProxSubdifferential}
\bm{p} = \mathrm{prox}_{f}(\bm{x}) \Leftrightarrow \bm{x} \in (\mathrm{Id}+\partial f)(\bm{p}).
\end{align}
We say that $f \in \Gamma_0(\mathbb{R}^{n})$ is prox-friendly if
$\mathrm{prox}_{\gamma f}$ is available as a computational tool for any $\gamma \in \mathbb{R}_{++}$.
Note that, if $f \in \Gamma_0(\mathbb{R}^{n})$ is prox-friendly,
then $f^{*}$ is also prox-friendly because the following relation holds
for any $\gamma \in \mathbb{R}_{++}$ and $\bm{x} \in \mathbb{R}^{n}$ \cite[Theorem 14.3(ii)]{BC:ConvexAnalysis}:
\begin{align*}
\mathrm{prox}_{\gamma f^{*}}(\bm{x}) = \bm{x} - \gamma \mathrm{prox}_{\frac{1}{\gamma}f}
\left(\bm{x}/\gamma\right).
\end{align*}
The indicator function of a nonempty closed convex set $C \subset \mathbb{R}^{n}$ is defined by
\begin{align*}
\iota_{C}\colon\mathbb{R}^{n}\rightarrow\mathbb{R}\cup \{\infty \}
\colon\bm{x}\mapsto
\begin{cases}
0, & \text{\rm if }\bm{x} \in C;\\
\infty, &\text{\rm otherwise},
\end{cases}
\end{align*}
and $\iota_{C} \in \Gamma_0(\mathbb{R}^{n})$.
For any $\gamma \in \mathbb{R}_{++}$,
the proximity operator of $\gamma\iota_{C}$ reduces to the projection onto $C$, i.e.,
for every $\bm{x} \in \mathbb{R}^{n}$,
\begin{align*}
\mathrm{prox}_{\gamma\iota_{C}}(\bm{x}) = P_{C}(\bm{x}) \coloneqq 
\argmin_{\bm{y} \in C}\|\bm{x}-\bm{y}\|.
\end{align*}

\subsection{Fixed Point Theory of Nonexpansive Operators}
Let $(\mathcal{H},\langle\cdot,\cdot\rangle_{\mathcal{H}},\|\cdot\|_{\mathcal{H}})$ be a finite-dimensional real Hilbert space.
The fixed point set of an operator $T\colon\mathcal{H}\rightarrow\mathcal{H}$
is denoted by $\mathrm{Fix}(T) \coloneqq  \{\bm{x}\in\mathcal{H}\mid T(\bm{x}) = \bm{x}\}$.
An operator $T\colon\mathcal{H}\rightarrow\mathcal{H}$ is called nonexpansive if
\begin{align*}
(\forall \bm{x}\in\mathcal{H})(\forall \bm{y}\in\mathcal{H}) \quad \|T(\bm{x})-T(\bm{y}) \|_{\mathcal{H}} \leq \|\bm{x}-\bm{y}\|_{\mathcal{H}}.
\end{align*}
A nonexpansive operator
$T\colon\mathcal{H}\rightarrow\mathcal{H}$ is $\alpha$-averaged if there exist $\alpha \in (0,1)$
and a nonexpansive operator $R\colon\mathcal{H}\rightarrow\mathcal{H}$ such that
$T = (1-\alpha)\mathrm{Id} + \alpha R$, where $\mathrm{Id}$ is the identity operator.
\begin{fact}[Composition of averaged operators
{\cite[Proposition 4.44]{BC:ConvexAnalysis}}]
\label{fact:CompositionAveragedOp}
Let $T_{i}\colon\mathcal{H}\rightarrow\mathcal{H}$ be $\alpha_{i}$-averaged nonexpansive for each $i = 1, 2$.
Then, $T_{1}\circ T_{2}$ is $\alpha$-averaged nonexpansive with
$\alpha = (\alpha_{1}+\alpha_{2}-2\alpha_{1}\alpha_{2})/(1-\alpha_{1}\alpha_{2}) \in (0,1)$.
\end{fact}
\begin{fact}[Krasnosel'ski\u{\i}-Mann Iteration {\cite[Theorem 5.15]{BC:ConvexAnalysis}}]
\label{fact:KM_Iteration}
Let $T\colon\mathcal{H}\rightarrow\mathcal{H}$ be a nonexpansive operator with $\mathrm{Fix}(T) \neq \varnothing$.
For any $(\mu_k)_{k\in\mathbb{N}} \subset [0,1]$ satisfying $\sum_{k\in\mathbb{N}}\mu_k(1-\mu_k) = \infty$
and an initial point $\bm{x}_0 \in \mathcal{H}$, the sequence $(\bm{x}_k)_{k\in\mathbb{N}}$ generated by
\begin{align*}
(\forall k \in \mathbb{N})\quad \bm{x}_{k+1} = (1-\mu_k)\bm{x}_{k} + \mu_k T(\bm{x}_{k})
\end{align*}
converges\footnote{Since this paper focuses on finite-dimensional real Hilbert spaces, weak convergence and strong convergence are equivalent.} to a point in $\mathrm{Fix}(T)$. In particular, if $T$ is $\alpha$-averaged nonexpansive,
the sequence $(\bm{x}_k)_{k\in\mathbb{N}}$ generated by
\begin{align*}
(\forall k \in \mathbb{N})\quad \bm{x}_{k+1} = T(\bm{x}_{k})
\end{align*}
converges to a point in $\mathrm{Fix}(T)$.
\end{fact}

A set-valued operator $A\colon\mathcal{H}\rightarrow2^{\mathcal{H}}$
is called monotone if
\begin{align*}
(\forall \bm{x}\in\mathcal{H})(\forall \bm{y}\in\mathcal{H})(\forall \bm{u}\in A(\bm{x}))(\forall \bm{v}\in A(\bm{y}))\quad
\langle\bm{x}-\bm{y},\bm{u}-\bm{v}\rangle_{\mathcal{H}} \geq 0.
\end{align*}
A monotone operator $A\colon\mathcal{H}\rightarrow2^{\mathcal{H}}$ is maximally monotone if
there exists no monotone operator whose graph strictly contains the graph of $A$.
A set-valued operator $A\colon\mathcal{H}\rightarrow2^{\mathcal{H}}$ is maximally monotone
if and only if its resolvent $(\mathrm{Id}+A)^{-1}\colon\mathcal{H}\rightarrow2^{\mathcal{H}}$ is single-valued and 
$1/2$-averaged nonexpansive
\cite[Proposition 23.10]{BC:ConvexAnalysis}.
An important example of a maximally monotone operator
is the subdifferential $\partial f$ of $f \in \Gamma_0(\mathcal{H})$
\cite[Theorem 20.25]{BC:ConvexAnalysis}.
The proximity operator $\mathrm{prox}_{f}$
is the resolvent of $\partial f$ (see \eqref{eq:relationProxSubdifferential}),
and thus is $1/2$-averaged nonexpansive.

\section{Design of GME-MI Model}
\label{sect:design_GME_MI_penalty_model}
We consider a general linear inverse problem of estimating an original signal $\bm{x}_{\text{\rm org}} \in \mathcal{C}$
from the observed vector $\bm{y} \coloneqq  \bm{A}\bm{x}_{\text{\rm org}} +  \bm{\varepsilon}\in \mathbb{R}^{d}$,
where $\mathcal{C} \subset \mathbb{R}^{n}$ is a known nonempty closed convex constraint set,
$\bm{A}\in \mathbb{R}^{d \times n}$ is a known measurement matrix, and
$\bm{\varepsilon} \in \mathbb{R}^{d}$ is an unknown noise vector.

The proposed GME-MI model is designed as follows.
In Section \ref{sect:class_MI_penalty}, we introduce a class
of MI penalties, which covers the important existing examples: the LOP-$\ell_2/\ell_1$ and TGV penalties.
Then, in Section \ref{sect:design_GMEMI_penalty}, to remedy the underestimation tendencies of MI penalties,
we design the GME-MI
penalties as differences between MI penalties and their generalized Moreau envelopes.
Finally, in Section \ref{sect:design_GMEMI_Regularization_Model},
we design the GME-MI model by applying the GME-MI penalty to regularized least-squares estimation,
and derive its overall convexity condition.

\subsection{Class of MI Penalties}
\label{sect:class_MI_penalty}
Before defining the MI penalty function for $\bm{u} = \bm{L}\bm{x} \in \mathbb{R}^{m}$
with a known matrix $\bm{L}\in\mathbb{R}^{m \times n}$,
we introduce a \emph{seed function} $\varphi \in\Gamma_0(\mathbb{R}^{m}\times\mathbb{R}^{l})$
satisfying the following mild condition.
\begin{asmp}\label{Assumption:phi_Coresive_Proper_BoundedBelow}
Assume that $\varphi$ is bounded below, and $\varphi(\bm{u},\cdot)$ is proper and coercive
for every $\bm{u} \in \mathbb{R}^{m}$.
\end{asmp}
\begin{lemm}\label{lemm:nonempty_partialmin_phi}
For every $\bm{u} \in \mathbb{R}^{m}$, $ \argmin_{\bm{\sigma}\in\mathbb{R}^{l}}\varphi(\bm{u},\bm{\sigma}) \neq \varnothing$.
\end{lemm}
\begin{proof}
The properness of $\varphi(\bm{u},\cdot)$
and $\varphi \in \Gamma_0(\mathbb{R}^{m}\times\mathbb{R}^{l})$
imply $\varphi(\bm{u},\cdot) \in \Gamma_0(\mathbb{R}^{l})$, and thus its coercivity yields
the claim by \cite[Proposition 11.15]{BC:ConvexAnalysis}.
\end{proof}

\begin{defn}
\label{defn:MI_Penalty}
Given a seed function $\varphi \in\Gamma_0(\mathbb{R}^{m}\times\mathbb{R}^{l})$
satisfying Assumption \ref{Assumption:phi_Coresive_Proper_BoundedBelow}, we define the MI penalty by
\begin{align}
\label{eq:def:MI_penalty}
\psi\colon\mathbb{R}^{m}\rightarrow\mathbb{R}\colon\bm{u} \mapsto \min_{\bm{\sigma}\in\mathbb{R}^{l}}\varphi(\bm{u},\bm{\sigma}).
\end{align}
\end{defn}
\begin{lemm}\label{lemm:psi_lsc_convex_fullDomain}
$\mathrm{dom}(\psi) = \mathbb{R}^{m}$,
$\psi \in \Gamma_0(\mathbb{R}^{m})$,
and $\psi$ is bounded below.
\end{lemm}
\begin{proof}
Lemma \ref{lemm:nonempty_partialmin_phi} implies
$\min_{\bm{\sigma}\in\mathbb{R}^{l}}\varphi(\bm{u},\bm{\sigma}) \in \mathbb{R}$ for every $\bm{u} \in \mathbb{R}^{m}$,
and thus we have $\mathrm{dom}(\psi) = \mathbb{R}^{m}$.
Since $\psi$ is defined by the partial minimization of $\varphi$,
the convexity of $\psi$ follows from 
the convexity of $\varphi$ by \cite[Proposition 8.35]{BC:ConvexAnalysis}.
The convexity of $\psi$ and $\mathrm{dom}(\psi) = \mathbb{R}^{m}$ imply that
$\psi$ is continuous by \cite[Corollary 8.40]{BC:ConvexAnalysis}.
Thus, we have $\psi \in \Gamma_0(\mathbb{R}^{m})$.
Since $\varphi$ is bounded below by Assumption \ref{Assumption:phi_Coresive_Proper_BoundedBelow},
so is $\psi$.
\end{proof}

The introduced class of MI penalties covers the important existing penalties given in Example \ref{exmp:MI_penalties} below,
where Assumption \ref{Assumption:phi_Coresive_Proper_BoundedBelow} is automatically satisfied
(see \ref{proof:validity_phi_CoresiveAndProper}).

\begin{exmp}[Instances of MI penalties]\hfill
\label{exmp:MI_penalties}
\begin{enumerate}
\item[a)] (LOP-$\ell_2/\ell_1$)
Consider the problem of estimating $\bm{x}_{\text{\rm org}}$ that
is block-sparse in the range space of $\bm{W}  \in \mathbb{R}^{m\times n}$
with an unknown block partition, which arises in many applications \cite{Yu:Audio,Cevher:Video,Gao:Video,Hur:Communications,Kuroda_FPC,Kitahara:PAWR_GRSS,Kuroda:PSDEstimation}.
To leverage block-sparsity without the knowledge of concrete block partition,
the LOP-$\ell_2/\ell_1$ penalty \cite{Kuroda:BlockSparse}
is designed as a tight convex approximation of the combinatorial penalty that minimizes the mixed $\ell_2/\ell_1$ norm
of $\bm{W}\bm{x}$ over the set of possible block partitions.
By setting $\bm{L} = \bm{W}$,
the LOP-$\ell_2/\ell_1$ penalty
is reproduced by \eqref{eq:def:MI_penalty} for $\varphi_{\alpha}^{\text{\rm LOP}}\in\Gamma_0(\mathbb{R}^{m}\times\mathbb{R}^{m})$ defined by\footnote{While extension to complex-valued signals is straightforward, we focus on the real-valued signals to simplify the presentation.}
\begin{align}
\label{eq:def:seed_function_LOPl2l1}
\varphi_{\alpha}^{\text{\rm LOP}}(\bm{u},\bm{\sigma}) \coloneqq  
\sum_{i=1}^{m}h(u_i,\sigma_i)
+\iota_{B_1^{\alpha}}(\bm{D}\bm{\sigma}),
\end{align}
where $h \in\Gamma_0(\mathbb{R}\times\mathbb{R})$ is the \emph{perspective} \cite{Combettes:Perspective} of
$(1/2)u^2+1/2$, $\iota_{B_1^{\alpha}}\in\Gamma_0(\mathbb{R}^{p})$ is the indicator function of the $\ell_1$ norm ball,
$\bm{D}\in\mathbb{R}^{p \times m}$ is the first-order
discrete difference operator, i.e.,
\begin{align}
\label{eq:def:perspective_quadratic_func}
h(u,\sigma) &\coloneqq  \begin{dcases}
\frac{u^2}{2\sigma} + \frac{\sigma}{2}, &\text{\rm if }\sigma > 0;\\
0, &\mbox{{\rm if} } u=0 \text{ {\rm and} } \sigma = 0;\\
\infty, &\text{\rm otherwise},
\end{dcases}\\
\label{eq:def:indicator_func_l1_ball}
\iota_{B_1^{\alpha}}(\bm{\xi}) &\coloneqq 
\begin{cases}
0, &\text{\rm if }\|\bm{\xi}\|_1 \leq \alpha;\\
\infty, &\text{\rm otherwise},
\end{cases}\\
\nonumber
\bm{D}\bm{\sigma} &\coloneqq  (\sigma_{i_{1}}-\sigma_{i_{2}})_{(i_{1},i_{2})\in\mathcal{I}},
\end{align}
$\mathcal{I}$ is the set of indices of neighboring pairs, $p = |\mathcal{I}|$,
and $\alpha \in \mathbb{R}_{+}$ is a tuning parameter related to the number of blocks.
See \cite[Section II]{Kuroda:BlockSparse}
for how block partition is optimized with the latent vector $\bm{\sigma}$
using the minimization \eqref{eq:def:MI_penalty} for $\varphi = \varphi_{\alpha}^{\text{\rm LOP}}$.

\item[b)] (TGV) Consider the problem of estimating piecewise smooth $\bm{x}_{\text{\rm org}}$.
The TGV penalty \cite{Bredies:TGV} is introduced to
leverage piecewise smoothness by penalizing the magnitudes of discontinuous jumps of high-order derivatives.
Since second-order TGV is commonly used in many applications
\cite{Knoll:TGV,Valkonen:TGV,Ferstl:TGV,Ono:VTGV,Langkammer:TGV,Bredies:TGV_Manifold,Bredies:TGV_Review}, we focus on it to simplify the presentation.
The second-order TGV penalty is defined for $\bm{x} \in \mathbb{R}^n$ by
\begin{align*}
\mathrm{TGV}_{\alpha}(\bm{x})\coloneqq  
\min_{\bm{\sigma}\in\mathbb{R}^{m}}
\left[\alpha\|\bm{D}\bm{x}-\bm{\sigma}\|_{2,1}^{\mathcal{G}_{1}}
+(1-\alpha) \|\tilde{\bm{D}}\bm{\sigma}\|_{2,1}^{\mathcal{G}_{2}}\right],
\end{align*}
where $\bm{D}\in\mathbb{R}^{m \times n}$
is the first-order discrete difference operator,
$\tilde{\bm{D}}\in\mathbb{R}^{p \times m}$ is the \emph{symmetrized version} \cite{Bredies:TGV,Bredies:TGV_Manifold} of $\bm{D}$, $\mathcal{G}_{1}$ and $\mathcal{G}_{2}$ are partitions of $\{1,\ldots,m\}$ and $\{1,\ldots,p\}$, respectively, 
and $\alpha \in (0,1)$ is a tuning parameter.
Examples of $\mathcal{G}_{1}$ and $\mathcal{G}_{2}$ are provided below.
\begin{enumerate}
\item[i)] (Anisotropic TGV) We set $\mathcal{G}_{1} = (\{i\})_{i = 1}^{m}$ and $\mathcal{G}_{2} = (\{i\})_{i = 1}^{p}$,
i.e., $\|\cdot\|_{2,1}^{\mathcal{G}_{1}} = \|\cdot\|_{1}$ and $\|\cdot\|_{2,1}^{\mathcal{G}_{2}} = \|\cdot\|_{1}$.
\item[ii)] (Isotropic TGV)
For multi-dimensional signals such as images,
the isotropic version is often preferred \cite{Bredies:TGV_Review}, where the groups in $\mathcal{G}_{1}$ and  $\mathcal{G}_{2}$ are
designed by collecting the directional differences.
\end{enumerate}
By setting $\bm{L} = \bm{D}$,
the TGV penalty is reproduced by \eqref{eq:def:MI_penalty} for $\varphi_{\alpha}^{\text{\rm TGV}} \in \Gamma_0(\mathbb{R}^{m}\times\mathbb{R}^{m})$ defined by
\begin{align}
\label{eq:def:seed_function_TGV}
\varphi_{\alpha}^{\text{\rm TGV}}&(\bm{u},\bm{\sigma})
\coloneqq  \alpha\|\bm{u}-\bm{\sigma}\|_{2,1}^{\mathcal{G}_{1}}
+(1-\alpha) \|\tilde{\bm{D}}\bm{\sigma}\|_{2,1}^{\mathcal{G}_{2}}.
\end{align}
See \cite[Section 3.2]{Bredies:TGV} for how 
the jumps of the zeroth derivative (i.e.,~the signal itself)
and the first derivative of $\bm{x}$
are computed and penalized
using the minimization \eqref{eq:def:MI_penalty} for $\varphi = \varphi_{\alpha}^{\text{\rm TGV}}$,
where $\bm{u} = \bm{L}\bm{x}$ with $\bm{L} = \bm{D}$.
\end{enumerate}
\end{exmp}

As illustrated by Example \ref{exmp:MI_penalties},
the MI penalty is able to automatically incorporate a certain signal structure (e.g., blocks and derivatives)
by the minimization involved in its definition; however,
it often underestimates the magnitudes of significant components
because the value of the MI penalty usually increases with magnitude.
In particular, the LOP-$\ell_2/\ell_1$ penalty tends to underestimate blocks of large-magnitude components
because its value increases as the magnitude of components increases,
owing to its coercivity \cite[Theorem 1]{Kuroda:BlockSparse}.
Similarly, the TGV penalty tends to underestimate large jumps of derivatives
because it is a semi-norm \cite{Bredies:TGV}, i.e.,
its value increases as the magnitude of discontinuous jumps increases (see Fig.~\ref{fig:GME_TGV_penalty_value}(a) for an example).

\begin{figure}[t]
  \centering
    \includegraphics[width=0.9\columnwidth]{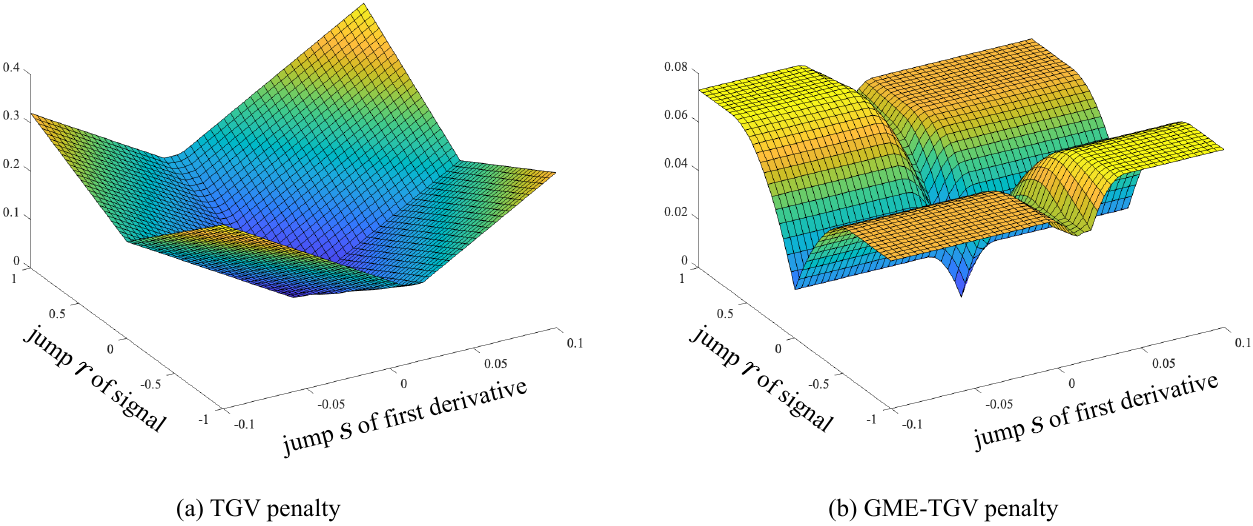}
  \caption{Values of TGV and GME-TGV penalties for $\bm{u}=\bm{D}\bm{x}$
   are plotted with respect to discontinuous jumps of a piecewise linear signal $\bm{x} \in \mathbb{R}^{50}$ defined by $x_i = 0$ $(i=1,2,\ldots,25)$ and $x_i = s(i-26)+r$ $(i=26,27,\ldots,50)$, where the parameters are set to be $\alpha = 0.2$ and $\bm{B} = \bm{I}$.}
  \label{fig:GME_TGV_penalty_value}
\end{figure}

\subsection{Design of GME-MI Penalty}
\label{sect:design_GMEMI_penalty}
To remedy the underestimation tendency of the MI penalty $\psi$,
we design the GME-MI penalty by subtracting from $\psi$ its generalized Moreau envelope.
Our design is inspired by
recent studies
\cite{Selesnick:GMC,Yin:CPNC_PCP,AlShabili:SRS,Lanza:GMC,Abe:LiGME,Yukawa:LiMES,Kitahara:LiGME_MRI,Yata:cLiGME}
on non-MI penalties (i.e., those defined without minimization) that have demonstrated that subtraction of the generalized Moreau envelopes from the respective penalties
mitigates the increase of the penalty values, and thus
remedies underestimation tendencies.
\begin{defn}
\label{defn:GMEMI_Penalty}
Based on the MI penalty $\psi$ in Definition \ref{defn:MI_Penalty},
we define the GME-MI penalty by
\begin{align}
\label{eq:def:GMEMI_penalty}
\Psi_{\bm{B}}\colon\mathbb{R}^{m}\rightarrow\mathbb{R}\cup\{\infty\}\colon
\bm{u} \mapsto  \psi(\bm{u}) - \inf_{\bm{v} \in \mathbb{R}^{m}}\left[\psi(\bm{v}) + \frac{1}{2}\|\bm{B}(\bm{u}-\bm{v})\|^2\right],
\end{align}
where $\bm{B}\in \mathbb{R}^{q \times m}$ can be tuned to adjust the shape of $\Psi_{\bm{B}}$.
\end{defn}
\begin{lemm}\label{lemma:domain_GME_MI_Penalty}
Under Assumption \ref{Assumption:phi_Coresive_Proper_BoundedBelow},
$\mathrm{dom}(\Psi_{\bm{B}}) = \mathbb{R}^{m}$.
\end{lemm}
\begin{proof}
The claim holds since $\psi(\bm{u}) \in \mathbb{R}$ and
$\psi + (1/2)\|\bm{B}(\bm{u}-\cdot)\|^2$
is bounded below for every $\bm{u} \in \mathbb{R}^{m}$
by Lemma \ref{lemm:psi_lsc_convex_fullDomain}.
\end{proof}
Applying the construction \eqref{eq:def:GMEMI_penalty}
to the MI penalties considered in Example \ref{exmp:MI_penalties}, we derive novel 
nonconvex enhancements of the LOP-$\ell_2/\ell_1$ and TGV penalties.
\begin{exmp}[Instances of GME-MI penalties]
\label{exmp:GME_MI_penalties}
\end{exmp}
\begin{enumerate}
\item[a)] (GME-LOP-$\ell_2/\ell_1$)
Applying \eqref{eq:def:GMEMI_penalty} to the LOP-$\ell_2/\ell_1$ penalty 
$\psi_{\alpha}^{\text{\rm LOP}}$ defined by \eqref{eq:def:MI_penalty}
for the seed function $\varphi_{\alpha}^{\text{\rm LOP}}$ in \eqref{eq:def:seed_function_LOPl2l1},
we derive a novel GME-LOP-$\ell_2/\ell_1$ penalty $\Psi_{\bm{B}, \alpha}^{\text{\rm LOP}}$.
While it is difficult to express $\Psi_{\bm{B}, \alpha}^{\text{\rm LOP}}$ in a closed form in general,
to illustrate how the underestimation is remedied, we provide an approximate relation with
$\bm{B} = (1/\sqrt{\gamma})\bm{I}$
and $\gamma \in \mathbb{R}_{++}$ as follows. Since $\psi_{\alpha}^{\text{\rm LOP}}$ is an approximation of the mixed $\ell_2/\ell_1$ norm 
with the optimized block partition \cite{Kuroda:BlockSparse}, we expect
$\psi_{\alpha}^{\text{\rm LOP}}(\bm{u}) \approx\sum_{k=1}^{K^{\star}_{\alpha}}\sqrt{|\mathcal{B}_k^{\star}|}\|\bm{u}_{\mathcal{B}_k^{\star}} \|_2$,
where $(\mathcal{B}_k^{\star})_{k=1}^{K^{\star}_{\alpha}}$ denotes the optimized block partition.
By further assuming a similar approximation $\psi_{\alpha}^{\text{\rm LOP}}(\bm{v}) \approx\sum_{k=1}^{K^{\star}_{\alpha}}\sqrt{|\mathcal{B}_k^{\star}|}\|\bm{v}_{\mathcal{B}_k^{\star}} \|_2$ in
\eqref{eq:def:GMEMI_penalty},
the \emph{block minimax concave penalty} \cite{Zhang:MCP,Huang:Review} using the optimized block partition is approximately reproduced as a special instance of the GME-LOP-$\ell_2/\ell_1$ penalty, i.e.,
\begin{align*}
\Psi_{\frac{1}{\sqrt{\gamma}}\bm{I}, \alpha}^{\text{\rm LOP}}(\bm{u}) &\approx \sum_{k=1}^{K^{\star}_{\alpha}}
\rho_{\gamma, |\mathcal{B}_k^{\star}|}\left(\|\bm{u}_{\mathcal{B}_k^{\star}}\|_2\right),\\
\rho_{\gamma, b}\colon \mathbb{R}_{+}\rightarrow\mathbb{R}_{+}&\colon t \mapsto
\begin{dcases}
\sqrt{b} t-\frac{t^2}{2\gamma}, &\text{\rm if }t \leq \gamma \sqrt{b};\\
\frac{\gamma b}{2}, &\text{\rm otherwise}.
\end{dcases}
\end{align*}
Based on this example, as the value of the GME-LOP-$\ell_2/\ell_1$ penalty
does not increase once the magnitudes of nonzero blocks exceed certain thresholds,
the underestimation of blocks of large-magnitude components is considered to have been resolved (see also numerical examples in Section \ref{sect:Experiment_blocksparse}).
Note that, since $\psi_{0}^{\text{\rm LOP}}(\bm{u}) = \sqrt{m}\|\bm{u}\|_2$
and $\psi_{\infty}^{\text{\rm LOP}}(\bm{u}) = \|\bm{u}\|_1$ by \cite[Theorem 2]{Kuroda:BlockSparse},
the equalities hold
for the coarsest block partition and the finest block partition,
which are optimal choices when $\alpha = 0$ and $\alpha = \infty$, respectively:
\begin{align*}
\Psi_{\frac{1}{\sqrt{\gamma}}\bm{I}, 0}^{\text{\rm LOP}}(\bm{u}) &= \rho_{\gamma, m}(\|\bm{u}\|_{2}),\\
\Psi_{\frac{1}{\sqrt{\gamma}}\bm{I}, \infty}^{\text{\rm LOP}}(\bm{u}) &= \sum_{i=1}^{m}
\rho_{\gamma, 1}(|u_i|).
\end{align*}

\item[b)] (GME-TGV) Applying \eqref{eq:def:GMEMI_penalty} to the TGV penalty 
$\psi_{\alpha}^{\text{\rm TGV}}$ defined by \eqref{eq:def:MI_penalty}
for the seed function $\varphi_{\alpha}^{\text{\rm TGV}}$ in \eqref{eq:def:seed_function_TGV},
we derive a novel GME-TGV penalty $\Psi_{\bm{B}, \alpha}^{\text{\rm TGV}}$.
To illustrate how the underestimation is remedied,
we present numerically computed values of $\Psi_{\bm{B}, \alpha}^{\text{\rm TGV}}$ in Fig.~\ref{fig:GME_TGV_penalty_value}(b).
It is clear from Fig.~\ref{fig:GME_TGV_penalty_value}(b) that the curve of the GME-TGV penalty becomes
flat when the magnitudes of jumps of derivatives exceed certain thresholds, and thus
the underestimation of large jumps is considered to have been resolved (see also numerical examples in Section \ref{sect:Experiment_piecewiselinear}).
\end{enumerate}

\subsection{GME-MI Model and Its Convexity Condition}
\label{sect:design_GMEMI_Regularization_Model}
We apply the GME-MI penalty to regularized least-squares estimation
of $\bm{x}_{\text{\rm org}} \in \mathcal{C}$ from $\bm{y} = \bm{A}\bm{x}_{\text{\rm org}} +  \bm{\varepsilon}$ as follows.
\begin{defn}
\label{defn:GMEMI_Regularization_Model}
Using the GME-MI penalty $\Psi_{\bm{B}}$ in Definition \ref{defn:GMEMI_Penalty}, we define the GME-MI model by
\begin{align}
\label{eq:def:OverallRegularizationModel}
\minimize_{\bm{x} \in \mathcal{C}}J(\bm{x}) \coloneqq  \frac{1}{2}\|\bm{y}-\bm{A}\bm{x}\|^2 + \lambda \Psi_{\bm{B}}(\bm{L}\bm{x}),
\end{align}
where $\lambda \in \mathbb{R}_{++}$ is the regularization parameter.
\end{defn}
\begin{lemm}\label{lemm:domain_GME_MI_cost_function}
Under Assumption \ref{Assumption:phi_Coresive_Proper_BoundedBelow}, $\mathrm{dom}(J) = \mathbb{R}^{n}$.
\end{lemm}
\begin{proof}
It follows from Lemma \ref{lemma:domain_GME_MI_Penalty}.
\end{proof}

While $\Psi_{\bm{B}}$ is nonconvex in general as it is the difference of convex functions,
we show that a suitable choice of $\bm{B}$ can ensure convexity
of the overall cost function $J$.

\begin{theo}
\label{theo:ExpressCostFunc_SumOfConvex}
Under Assumption \ref{Assumption:phi_Coresive_Proper_BoundedBelow}, we have
\begin{align}
\label{eq:ExpressCostFunc_SumOfConvex}
J(\bm{x}) = \frac{1}{2}\langle\bm{x},\bm{Q}\bm{x} \rangle
- \langle\bm{y},\bm{A}\bm{x} \rangle + \frac{1}{2}\|\bm{y}\|^2 
 + \lambda\psi(\bm{L}\bm{x})
+ \lambda\left(\psi + \frac{1}{2}\|\bm{B}\cdot \|^2\right)^{*}(\bm{B}^{\top}\bm{B}\bm{L}\bm{x})
\end{align}
for every $\bm{x} \in\mathbb{R}^{n}$,
and $J \in \Gamma_{0}(\mathbb{R}^{n})$ if
\begin{align}
\label{eq:OverallConvexCondition}
\bm{Q} \coloneqq  \bm{A}^{\top}\bm{A} - \lambda \bm{L}^{\top}\bm{B}^{\top}\bm{B}\bm{L} \succeq \bm{O}.
\end{align}
\end{theo}
\begin{proof}
By Definitions \ref{defn:GMEMI_Penalty} and \ref{defn:GMEMI_Regularization_Model},
since
\begin{align*}
\frac{1}{2}\|\bm{B}(\bm{L}\bm{x}-\bm{v})\|^2
= \frac{1}{2}\|\bm{B}\bm{L}\bm{x}\|^2 -\langle\bm{B}^{\top}\bm{B}\bm{L}\bm{x},\bm{v} \rangle + \frac{1}{2}\|\bm{B}\bm{v}\|^2,
\end{align*}
and $\|\bm{B}\bm{L}\bm{x}\|^2$ is independent of $\bm{v}$, we obtain
\begin{align*}
J(\bm{x}) &=\frac{1}{2}\langle\bm{x},\bm{Q}\bm{x} \rangle -\langle\bm{y},\bm{A}\bm{x} \rangle + \frac{1}{2}\|\bm{y}\|^2  +\lambda \psi(\bm{L}\bm{x}) - \lambda \inf_{\bm{v} \in \mathbb{R}^{m}}\left[\psi(\bm{v})
+\frac{1}{2}\|\bm{B}\bm{v}\|^2
-\langle\bm{B}^{\top}\bm{B}\bm{L}\bm{x},\bm{v} \rangle
\right],
\end{align*}
where we use the equality $\langle\bm{x},\bm{Q}\bm{x} \rangle = \|\bm{A}\bm{x}\|^2 - \lambda \|\bm{B}\bm{L}\bm{x}\|^2$ by the definition of $\bm{Q}$ in \eqref{eq:OverallConvexCondition}.
Moreover, since
\begin{align*}
&- \lambda \inf_{\bm{v} \in \mathbb{R}^{m}}\left[\psi(\bm{v})
+\frac{1}{2}\|\bm{B}\bm{v}\|^2
-\langle\bm{B}^{\top}\bm{B}\bm{L}\bm{x},\bm{v} \rangle
\right]\\
={}&\lambda\sup_{\bm{v} \in \mathbb{R}^{m}}\left[-\psi(\bm{v})
-\frac{1}{2}\|\bm{B}\bm{v}\|^2
+\langle\bm{B}^{\top}\bm{B}\bm{L}\bm{x},\bm{v} \rangle
\right]\\
={}&\lambda \left(\psi + \frac{1}{2}\|\bm{B}\cdot \|^2\right)^{*}(\bm{B}^{\top}\bm{B}\bm{L}\bm{x})
\end{align*}
due to the definition \eqref{eq:def:conjugate_convexFunc} of the conjugate,
we have the expression \eqref{eq:ExpressCostFunc_SumOfConvex}.
By Lemma \ref{lemm:domain_GME_MI_cost_function} and the expression \eqref{eq:ExpressCostFunc_SumOfConvex},
$\lambda (\psi + (1/2)\|\bm{B}\cdot \|^2)^{*} \circ \bm{B}^{\top}\bm{B}\bm{L}$
is proper. Thus,
from $\psi \in \Gamma_0(\mathbb{R}^{m})$ by Lemma \ref{lemm:psi_lsc_convex_fullDomain},
we have
$\lambda (\psi + (1/2)\|\bm{B}\cdot \|^2)^{*} \circ \bm{B}^{\top}\bm{B}\bm{L} \in \Gamma_0(\mathbb{R}^{n})$.
Meanwhile, $\psi \in \Gamma_0(\mathbb{R}^{m})$
and $\mathrm{dom}(\psi) = \mathbb{R}^{m}$ by Lemma \ref{lemm:psi_lsc_convex_fullDomain}
and the condition \eqref{eq:OverallConvexCondition} yield
that the function
\begin{align*}
\bm{x} \mapsto \frac{1}{2}\langle\bm{x},\bm{Q}\bm{x} \rangle -\langle\bm{y},\bm{A} \bm{x} \rangle + \frac{1}{2}\|\bm{y}\|^2  +\lambda \psi(\bm{L}\bm{x})
\end{align*}
belongs to $\Gamma_0(\mathbb{R}^{n})$.
Altogether, since $J$ is proper by Lemma \ref{lemm:domain_GME_MI_cost_function}, $J \in \Gamma_{0}(\mathbb{R}^{n})$ is guaranteed.
\end{proof}

\begin{rmrk}[Design of \mbox{$\bm{B}$} satisfying \eqref{eq:OverallConvexCondition}]
\label{rmrk:Design_GMEMat_OverallConvex}
While this paper constructs the novel GME-MI model from the MI penalty $\psi$ defined via minimization in \eqref{eq:def:MI_penalty}, the condition \eqref{eq:OverallConvexCondition} is common with existing studies \cite{Abe:LiGME,Kitahara:LiGME_MRI,Yata:cLiGME,Chen:GMEmat} on enhancement of linearly involved prox-friendly penalties defined without minimization.
This allows using the results of the existing studies to set $\bm{B}$ satisfying \eqref{eq:OverallConvexCondition}. In particular, by \cite[Theorem 1]{Chen:GMEmat},
we can set $\bm{B}$ for any pair of $\bm{A}$ and $\bm{L}$
in the closed form via LU decomposition (see also Section \ref{sect:Experiment} for the specific configurations in experiments).
\end{rmrk}

\section{Optimization Algorithm for GME-MI Model}
\label{sect:OptimizationAlgorithm}
We present the proposed proximal splitting algorithm with guaranteed convergence to
a globally optimal solution of the GME-MI model in Definition \ref{defn:GMEMI_Regularization_Model}
under the overall convexity condition \eqref{eq:OverallConvexCondition}.
Although the expression \eqref{eq:ExpressCostFunc_SumOfConvex}
represents the cost function $J$ in \eqref{eq:def:OverallRegularizationModel} as the sum of convex functions under the condition \eqref{eq:OverallConvexCondition}, its minimization is still challenging because
it involves non-prox-friendly functions that are difficult to handle:
$\psi$ defined by minimization in \eqref{eq:def:MI_penalty}
and the conjugate of the sum of $\psi$ and $(1/2)\|\bm{B}\cdot \|^2$.
As a result,
the standard proximal splitting algorithms (see, e.g., \cite{Combettes:ProxSplit,Condat:ProxSplit})
as well as the special algorithms
\cite{Lanza:GMC,Abe:LiGME,Kitahara:LiGME_MRI,AlShabili:SRS,Yata:cLiGME,Yukawa:LiMES}
for GME models of (linearly involved) prox-friendly penalties are not applicable to the GME-MI model.
We resolve this challenge by carefully designing
a computable averaged nonexpansive operator $T_{\text{\rm GME-MI}}$
whose fixed point set characterizes the solution set 
\begin{align}
\label{eq:def:SolutionSetGMEMI}
\mathcal{S} \coloneqq  \argmin_{\bm{x} \in \mathcal{C}} J(\bm{x})
= \argmin_{\bm{x} \in \mathbb{R}^{n}} \left[J(\bm{x})+\iota_{\mathcal{C}}(\bm{x})\right].
\end{align}
After introducing technical assumptions and lemmas in Section \ref{sect:assumption_lemma_for_proposed_optim_alg},
in Section \ref{sect:convergence_optim_alg}, we derive $T_{\text{\rm GME-MI}}$
via careful reformulation of
$\bm{x}^{\star} \in \mathcal{S}\Leftrightarrow \bm{0} \in \partial(J+\iota_{\mathcal{C}})(\bm{x}^{\star})$
and establish its averaged nonexpansiveness,
which yields the convergence of the proposed algorithm based on
the Krasnosel'ski\u{\i}-Mann iteration (Fact \ref{fact:KM_Iteration}) of $T_{\text{\rm GME-MI}}$
to an optimal solution in $\mathcal{S}$.

\subsection{Technical Assumptions and Lemmas}
\label{sect:assumption_lemma_for_proposed_optim_alg}
We adopt rather simple assumptions to reduce technical complexity, although these assumptions could be relaxed.
We first impose the following condition on the seed function $\varphi$
and the constraint set $\mathcal{C}$.
\begin{asmp}
\label{asmp:form_seed_func_for_optimization}
Assume that the seed function $\varphi$ can be represented as
\begin{align}
\label{eq:represent_seed_function}
\varphi(\bm{u},\bm{\sigma}) = f(\bm{u},\bm{\sigma}) + g(\bm{M}\bm{\sigma}),
\end{align}
where $f\in\Gamma_0(\mathbb{R}^{m}\times\mathbb{R}^{l})$ and $g\in\Gamma_0(\mathbb{R}^{p})$
are prox-friendly, and
$\bm{M}\in \mathbb{R}^{p \times l}$.
We also assume that $\iota_{\mathcal{C}}$ is prox-friendly, i.e., the projection $P_{\mathcal{C}}$ onto $\mathcal{C}$ is computable.
\end{asmp}
While the form \eqref{eq:represent_seed_function} may seem rather specific,
it covers
the seed functions used for the important instances of the GME-MI penalty presented in Example \ref{exmp:GME_MI_penalties}.
\begin{exmp}\hfill
\label{exmp:representation_MIpenalty_forOptimization}
\begin{enumerate}
\item[a)] The seed function in \eqref{eq:def:seed_function_LOPl2l1} for the GME-LOP-$\ell_2/\ell_1$ penalty can be represented in the form of \eqref{eq:represent_seed_function} by setting
$f(\bm{u},\bm{\sigma}) = \sum_{i=1}^{m}h(u_i,\sigma_i)$,
$g = \iota_{B_1^{\alpha}}$, and $\bm{M} = \bm{D}$.
\item[b)] The seed function in \eqref{eq:def:seed_function_TGV} for the GME-TGV penalty can be represented
in the form of \eqref{eq:represent_seed_function}
by setting $f(\bm{u},\bm{\sigma}) = \alpha\|\bm{u}-\bm{\sigma}\|_{2,1}^{\mathcal{G}_{1}}$,
$g = (1-\alpha) \|\cdot\|_{2,1}^{\mathcal{G}_{2}}$, and
$\bm{M} = \tilde{\bm{D}}$.
\end{enumerate}
\ref{appendix:ComputeProxOp} shows that these $f$ and $g$ are prox-friendly. Note that the constraint set $\mathcal{C}$ is independent of the seed function,
and can be chosen based on the application.
\end{exmp}\noindent

To use the formulas of subdifferential calculus (Fact \ref{fact:SubdifferentialCalculus}) in the derivation of the proposed algorithm,
we further adopt the following assumption, which automatically holds
for the seed functions used for Example \ref{exmp:GME_MI_penalties}
with their representations in Example \ref{exmp:representation_MIpenalty_forOptimization}
(see \ref{proof:QualificationConditions}).
\begin{asmp}
\label{Assumption:QualificationConditions}
Define
\begin{align}
\label{eq:def:L_bar}
\bar{\bm{L}}&\coloneqq 
\begin{bmatrix}
\bm{L} & \bm{O}\\
\bm{O}& \bm{I}
\end{bmatrix} \in \mathbb{R}^{(m+l)\times(n+l)},\\
\label{eq:def:M_bar}
\bar{\bm{M}}&\coloneqq 
\begin{bmatrix}
\bm{O} & \bm{M}
\end{bmatrix}\in \mathbb{R}^{p \times(m+l)},
\end{align}
and assume that the following conditions hold.
\begin{enumerate}
\item[i)] One of the following conditions holds:\vspace{2pt}
\begin{enumerate}
\item[a)] $(\forall (\bm{u},\bm{\sigma})\in\mathbb{R}^{m}\times\mathbb{R}^{l}) \quad \varphi(-\bm{u},\bm{\sigma}) = \varphi(\bm{u},\bm{\sigma})$.\vspace{2pt}
\item[b)] $(\forall (\bm{u},\bm{\sigma})\in\mathbb{R}^{m}\times\mathbb{R}^{l}) \quad \varphi(-\bm{u},\bm{\sigma}) = \varphi(\bm{u},-\bm{\sigma})$.\vspace{2pt}
\end{enumerate}
\item[ii)] $\bm{0} \in \mathrm{ri}\left[\mathrm{dom}(\varphi) - \mathrm{ran}(\bar{\bm{L}})\right]$.\vspace{2pt}
\item[iii)] $\bm{0} \in \mathrm{ri}\left[\mathrm{dom}(g) -\bar{\bm{M}}(\mathrm{dom}(f))\right]$.\vspace{2pt}
\end{enumerate}
\end{asmp}
\begin{lemm}
\label{lemm:Qualification_conj_psi_squaredBNorm}
Assumptions \ref{Assumption:phi_Coresive_Proper_BoundedBelow} and \ref{Assumption:QualificationConditions}(i) imply
\begin{align*}
\mbox{i)}&\quad \mathrm{dom}\left[\left(\psi + \frac{1}{2}\|\bm{B}\cdot \|^2\right)^{*} \circ \bm{B}^{\top}\bm{B}\bm{L}\right] = \mathbb{R}^{n},\\
\mbox{ii)}&\quad \bm{0} \in \mathrm{ri}\left[\mathrm{dom}\left(\left(\psi + \frac{1}{2}\|\bm{B}\cdot \|^2\right)^{*}\right) - \mathrm{ran}(\bm{B}^{\top}\bm{B}\bm{L})\right].
\end{align*}
\end{lemm}
\begin{proof}
Claim (i) follows from Lemma \ref{lemm:domain_GME_MI_cost_function} and the expression \eqref{eq:ExpressCostFunc_SumOfConvex}.
We prove claim (ii) as follows.
Assumption \ref{Assumption:QualificationConditions}(i)
and  Definition \ref{defn:MI_Penalty} imply
$\psi(-\bm{u}) = \psi(\bm{u})$ for every $\bm{u} \in \mathbb{R}^{m}$.
Thus, for any $\bm{w} \in \mathbb{R}^{m}$, we have
\begin{align*}
\left(\psi + \frac{1}{2}\|\bm{B}\cdot \|^2\right)^{*}(-\bm{w})
= \left(\psi + \frac{1}{2}\|\bm{B}\cdot \|^2\right)^{*}(\bm{w}),
\end{align*}
which implies that $\mathcal{A} \coloneqq  \mathrm{dom}\left[\left(\psi + \frac{1}{2}\|\bm{B}\cdot \|^2\right)^{*}\right]$ is
symmetric, i.e., $\mathcal{A} = -\mathcal{A}$.
Since $\left(\psi + \frac{1}{2}\|\bm{B}\cdot \|^2\right)^{*} \in \Gamma_0(\mathbb{R}^{m})$
is proved by Theorem \ref{theo:ExpressCostFunc_SumOfConvex},
$\mathcal{A}$ is nonempty and convex by \cite[Proposition 8.2]{BC:ConvexAnalysis}.
Altogether, $\mathcal{A} - \mathrm{ran}(\bm{B}^{\top}\bm{B}\bm{L})$
is a nonempty symmetric convex set, and thus claim (ii) holds by \cite[Example 6.10]{BC:ConvexAnalysis}.
\end{proof}

We also use the following lemma, which is a refinement of \cite[Proposition 16.59]{BC:ConvexAnalysis} for our scenario.
\begin{lemm}
\label{lemma:subdifferential_MarginalFunc}
For any $(\bm{u},\bm{\varsigma})\in\mathbb{R}^{m}\times\mathbb{R}^{l}$ and $\bm{r} \in \mathbb{R}^{m}$, we have
\begin{equation}
\label{eq:subdifferential_MarginalFunc_psi}
\left\{\begin{aligned}
&\bm{r} \in  \partial \psi(\bm{u})\\
&\bm{\varsigma} \in \argmin_{\bm{\sigma}\in\mathbb{R}^{l}}\varphi(\bm{u},\bm{\sigma})
\end{aligned}\right.
\Leftrightarrow
(\bm{r},\bm{0}) \in \partial \varphi(\bm{u},\bm{\varsigma}).
\end{equation}
Similarly, for any $(\bm{x},\bm{\varsigma})\in\mathbb{R}^{n}\times\mathbb{R}^{l}$ and $\bm{s} \in \mathbb{R}^{n}$, we have
\begin{equation}
\label{eq:subdifferential_MarginalFunc_psi_L}
\left\{\begin{aligned}
&\bm{s} \in  \partial (\psi \circ\bm{L})(\bm{x})\\
&\bm{\varsigma} \in \argmin_{\bm{\sigma}\in\mathbb{R}^{l}}\varphi(\bm{L}\bm{x},\bm{\sigma})
\end{aligned}\right.
\Leftrightarrow
(\bm{s},\bm{0}) \in \partial (\varphi\circ\bar{\bm{L}})(\bm{x},\bm{\varsigma}).
\end{equation}
\end{lemm}
\begin{proof}
First, we prove the relation \eqref{eq:subdifferential_MarginalFunc_psi}.
\begin{enumerate}
\item[$(\Rightarrow)$]
The condition $\bm{\varsigma} \in \argmin_{\bm{\sigma}\in\mathbb{R}^{l}}\varphi(\bm{u},\bm{\sigma})$
and Definition \ref{defn:MI_Penalty} imply $\psi(\bm{u}) = \varphi(\bm{u},\bm{\varsigma})$.
Thus, since $\psi$ is proper by Lemma \ref{lemm:psi_lsc_convex_fullDomain}
and $\varphi \in\Gamma_0(\mathbb{R}^{m}\times\mathbb{R}^{l})$,
the claim follows from \cite[Proposition 16.59]{BC:ConvexAnalysis}.
\item[$(\Leftarrow)$] Suppose $(\bm{r},\bm{0}) \in \partial \varphi(\bm{u},\bm{\varsigma})$.
Then, for any $(\bm{v},\bm{\sigma}) \in \mathbb{R}^{m}\times\mathbb{R}^{l}$, we have
\begin{align*}
&\langle(\bm{v},\bm{\sigma})-(\bm{u},\bm{\varsigma}),(\bm{r},\bm{0}) \rangle + \varphi(\bm{u},\bm{\varsigma}) \leq \varphi(\bm{v},\bm{\sigma})\\
\Leftrightarrow {}&  \langle\bm{v}-\bm{u},\bm{r} \rangle + \langle\bm{\sigma}-\bm{\varsigma},\bm{0} \rangle +  \varphi(\bm{u},\bm{\varsigma}) \leq \varphi(\bm{v},\bm{\sigma})\\
\Leftrightarrow {}& \langle\bm{v}-\bm{u},\bm{r} \rangle +  \varphi(\bm{u},\bm{\varsigma}) \leq \varphi(\bm{v},\bm{\sigma}).
\end{align*}
In particular, by setting $\bm{v} = \bm{u}$, we obtain
$(\forall \bm{\sigma}\in\mathbb{R}^{l})\quad \varphi(\bm{u},\bm{\varsigma}) \leq \varphi(\bm{u},\bm{\sigma})$,
i.e., $\bm{\varsigma} \in \argmin_{\bm{\sigma}\in\mathbb{R}^{l}}\varphi(\bm{u},\bm{\sigma})$,
and thus $\varphi(\bm{u},\bm{\varsigma}) = \psi(\bm{u})$.
This equality and
$(\bm{r},\bm{0}) \in \partial \varphi(\bm{u},\bm{\varsigma})$
imply $\bm{r} \in  \partial \psi(\bm{u})$ by \cite[Proposition 16.59]{BC:ConvexAnalysis}.
\end{enumerate}
The relation \eqref{eq:subdifferential_MarginalFunc_psi_L} can be derived similarly, since
the definition of $\bar{\bm{L}}$ in \eqref{eq:def:L_bar} yields
$(\psi\circ \bm{L})(\bm{x}) = \min_{\bm{\sigma}\in\mathbb{R}^{l}}\varphi(\bm{L}\bm{x},\bm{\sigma})
=\min_{\bm{\sigma}\in\mathbb{R}^{l}}\left(\varphi\circ\bar{\bm{L}}\right)(\bm{x},\bm{\sigma})$.
\end{proof}

\subsection{Design of Proposed Algorithm and Its Convergence}
\label{sect:convergence_optim_alg}
Now we derive the averaged nonexpansive operator $T_{\text{\rm GME-MI}}$
that characterizes the solution set $\mathcal{S}$ in \eqref{eq:def:SolutionSetGMEMI}
and establish the convergence of Algorithm \ref{alg:Optimization_GME_MI},
designed based on
the Krasnosel'ski\u{\i}-Mann iteration of $T_{\text{\rm GME-MI}}$,
to a globally optimal solution of the GME-MI model in Definition \ref{defn:GMEMI_Regularization_Model}.

\begin{theo}
\label{theo:Convergence_OptimizationAlg}
Set $\kappa > 1$ and $\gamma_1,\gamma_2,\gamma_3, \gamma_4 \in \mathbb{R}_{++}$
satisfying\footnote{%
For any $\kappa > 1$, the condition \eqref{eq:cond_stepSizeLikeParam} is satisfied by, e.g., using $\delta \in \mathbb{R}_{++}$,
$\gamma_1 = 1/(\|(\kappa/2)\bm{A}^{\top}\bm{A}+\lambda\bm{L}^{\top}\bm{L}\|_{\mathrm{op}} + \delta)$,
$\gamma_2 = 1/(\|\bm{M}\|_{\mathrm{op}}^2 + 1 + \delta)$,
$\gamma_3 =  1/((\kappa/2 + 2/\kappa)\|\bm{B}\|_{\mathrm{op}}^2+\delta)$, and
$\gamma_4 = 1/(\gamma_3\|\bm{M}\|_{\mathrm{op}}^2 + \delta)$.}
\begin{equation}
\label{eq:cond_stepSizeLikeParam}
\left\{\begin{aligned}
&\frac{1}{\gamma_1}\bm{I} - \frac{\kappa}{2}\bm{A}^{\top}\bm{A}-\lambda\bm{L}^{\top}\bm{L} \succ \bm{O},\\
&\frac{1}{\gamma_2}\bm{I} - \bm{I} - \bm{M}^{\top}\bm{M} \succ \bm{O},\\
&\frac{1}{\gamma_3} \geq  \left(\frac{\kappa}{2} + \frac{2}{\kappa}\right)\|\bm{B}\|_{\mathrm{op}}^2,\\
&\frac{1}{\gamma_4}\bm{I} - \gamma_3\bm{M}\bm{M}^{\top} \succ \bm{O}.\\
\end{aligned}\right.
\end{equation}
Let $\mathcal{H} \coloneqq  \mathbb{R}^{n}\times\mathbb{R}^{l}\times\mathbb{R}^{m}\times\mathbb{R}^{l}\times\mathbb{R}^{m}\times\mathbb{R}^{l}\times\mathbb{R}^{p}\times\mathbb{R}^{p}$.
Define $T_{\text{\rm GME-MI}}\colon\mathcal{H}\rightarrow\mathcal{H}\colon
(\bm{x},
\bm{\sigma},
\bm{v},
\bm{\tau},
\bm{r},
\bm{\eta},
\bm{\xi},
\bm{\zeta})\mapsto (\hat{\bm{x}},
\hat{\bm{\sigma}},
\hat{\bm{v}},
\hat{\bm{\tau}},
\hat{\bm{r}},
\hat{\bm{\eta}},
\hat{\bm{\xi}},
\hat{\bm{\zeta}})$
by
\begin{align}
\label{eq:def:AveragedOperatorForGME-MI}
\begin{aligned}
\hat{\bm{x}} &\coloneqq  P_{\mathcal{C}}[\bm{x}-\gamma_1(\bm{Q}\bm{x} - \bm{A}^{\top}\bm{y}+\lambda\bm{L}^{\top}\bm{B}^{\top}\bm{B}\bm{v}+\lambda\bm{L}^{\top}\bm{r})],\\
\hat{\bm{\sigma}} &\coloneqq  \bm{\sigma}-\gamma_2(\bm{\eta}+\bm{M}^{\top}\bm{\xi}),\\
(\hat{\bm{v}},\hat{\bm{\tau}})
&\coloneqq \mathrm{prox}_{\gamma_3 f}(
\bm{v}+\gamma_3\bm{B}^{\top}\bm{B}(\bm{L}(2\hat{\bm{x}}-\bm{x})-\bm{v}),
\bm{\tau}-\gamma_3\bm{M}^{\top}\bm{\zeta}),\\
(\hat{\bm{r}},\hat{\bm{\eta}}) &\coloneqq 
\mathrm{prox}_{f^{*}}(\bm{r} + \bm{L}(2\hat{\bm{x}}-\bm{x}),\bm{\eta}+2\hat{\bm{\sigma}}-\bm{\sigma}),\\
\hat{\bm{\xi}}&\coloneqq \mathrm{prox}_{g^{*}}(\bm{\xi}+\bm{M}(2\hat{\bm{\sigma}}-\bm{\sigma})),\\
\hat{\bm{\zeta}}&\coloneqq \mathrm{prox}_{\gamma_4 g^{*}}(\bm{\zeta} + \gamma_4\bm{M}(2\hat{\bm{\tau}}-\bm{\tau})).
\end{aligned}
\end{align}
Then, under Assumptions \ref{Assumption:phi_Coresive_Proper_BoundedBelow}, \ref{asmp:form_seed_func_for_optimization}, and \ref{Assumption:QualificationConditions} and the condition \eqref{eq:OverallConvexCondition},
\begin{align}
\label{eq:Characterization_SolutionSet_AveragedNonexpansiveOp}
\bm{x}^{\star} \in \mathcal{S}
\Leftrightarrow
\exists (\bm{\sigma}^{\star},
\bm{v}^{\star},
\bm{\tau}^{\star},
\bm{r}^{\star},
\bm{\eta}^{\star},
\bm{\xi}^{\star},
\bm{\zeta}^{\star})\text{ {\rm such that} } 
(\bm{x}^{\star},\bm{\sigma}^{\star},
\bm{v}^{\star},
\bm{\tau}^{\star},
\bm{r}^{\star},
\bm{\eta}^{\star},
\bm{\xi}^{\star},
\bm{\zeta}^{\star}) \in \mathrm{Fix}(T_{\text{\rm GME-MI}}),
\end{align}
and $T_{\text{\rm GME-MI}}$ is $\kappa/(2\kappa-1)$-averaged nonexpansive in
$\mathcal{H}$
with $\langle\cdot,\cdot\rangle_{\bm{P}}$ and $\|\cdot\|_{\bm{P}}$, where
$\bm{P} \in\mathbb{R}^{\mathrm{dim}\mathcal{H}\times\mathrm{dim}\mathcal{H}}$ is the positive definite matrix
defined by
\begin{equation}
\label{eq:def:positiveSelfAdjOp}
\begin{aligned}
\bm{P}
\begin{bmatrix}
\bm{x}\\
\bm{\sigma}\\
\bm{v}\\
\bm{\tau}\\
\bm{r}\\
\bm{\eta}\\
\bm{\xi}\\
\bm{\zeta}
\end{bmatrix}
\coloneqq 
\begin{bmatrix*}[l]
(1/\gamma_1)\bm{x}-\lambda\bm{L}^{\top}\bm{B}^{\top}\bm{B}\bm{v} - \lambda\bm{L}^{\top}\bm{r}\\
(\lambda/\gamma_2)\bm{\sigma}- \lambda\bm{\eta}- \lambda\bm{M}^{\top}\bm{\xi}\\
(\lambda/\gamma_3)\bm{v}-\lambda\bm{B}^{\top}\bm{B}\bm{L}\bm{x}\\
(\lambda/\gamma_3)\bm{\tau}-\lambda\bm{M}^{\top}\bm{\zeta}\\
\lambda\bm{r}-\lambda\bm{L}\bm{x}\\
\lambda\bm{\eta}-\lambda\bm{\sigma}\\
\lambda\bm{\xi}-\lambda\bm{M}\bm{\sigma}\\
(\lambda/\gamma_4)\bm{\zeta}-\lambda\bm{M}\bm{\tau}
\end{bmatrix*}.
\end{aligned}
\end{equation}
In conjunction with Fact \ref{fact:KM_Iteration}, these properties yield that $(\bm{x}_k)_{k\in\mathbb{N}}$ generated by Algorithm \ref{alg:Optimization_GME_MI}
converges to a point in $\mathcal{S}$, i.e., a globally optimal solution of \eqref{eq:def:OverallRegularizationModel},
if $\mathcal{S}$ is nonempty\footnote{%
$\mathcal{S} \neq \varnothing$ is automatically guaranteed in many scenarios,
e.g., if $\mathcal{C}$ is bounded.
Since $\mathrm{dom}(J) = \mathbb{R}^{n}$ holds by Lemma \ref{lemm:domain_GME_MI_cost_function}, we have $\mathrm{dom}(J) \cap \mathcal{C} \neq \varnothing$.
Thus, if $\mathcal{C}$ is bounded, then
$\mathcal{S} \neq \varnothing$ holds by \cite[Proposition 11.15]{BC:ConvexAnalysis}
under the condition \eqref{eq:OverallConvexCondition} guaranteeing $J \in \Gamma_{0}(\mathbb{R}^{n})$.
}.
\end{theo}
\begin{algorithm}[t]
\SetCommentSty{textrm}
\caption{Solver for GME-MI model \eqref{eq:def:OverallRegularizationModel} with 
\eqref{eq:def:MI_penalty}, \eqref{eq:def:GMEMI_penalty}, and \eqref{eq:represent_seed_function}}
\label{alg:Optimization_GME_MI}
\setstretch{1.4}
\KwIn{$\bm{B}$ satisfying \eqref{eq:OverallConvexCondition},
$\gamma_1,\gamma_2,\gamma_3,\gamma_4$ satisfying \eqref{eq:cond_stepSizeLikeParam},
$\bm{x}_{0}\in \mathbb{R}^{n},
\bm{\sigma}_{0}\in\mathbb{R}^{l},
\bm{v}_{0}\in \mathbb{R}^{m},
\bm{\tau}_{0} \in \mathbb{R}^{l},
\bm{r}_0 \in \mathbb{R}^{m},
\bm{\eta}_0 \in \mathbb{R}^{l},
\bm{\xi}_0 \in \mathbb{R}^{p},
\bm{\zeta}_0 \in \mathbb{R}^{p}$.}
\For{$k = 0,1,2,\ldots$}{
		$\bm{x}_{k+1} = P_{\mathcal{C}}[\bm{x}_k-\gamma_1(\bm{Q}\bm{x}_{k} - \bm{A}^{\top}\bm{y}+\lambda\bm{L}^{\top}(\bm{B}^{\top}\bm{B}\bm{v}_k+\bm{r}_k))]$\;
		$\bm{\sigma}_{k+1} = \bm{\sigma}_k-\gamma_2(\bm{\eta}_k+\bm{M}^{\top}\bm{\xi}_{k})$\;
		$\bm{u}_{k+1} = \bm{L}(2\bm{x}_{k+1}-\bm{x}_{k})$\;
		$(\bm{v}_{k+1},\bm{\tau}_{k+1})$ $=\mathrm{prox}_{\gamma_3 f}(\bm{v}_k+\gamma_3\bm{B}^{\top}\bm{B}(\bm{u}_{k+1}-\bm{v}_k),\bm{\tau}_k-\gamma_3\bm{M}^{\top}\bm{\zeta}_{k})$\;
		$(\bm{r}_{k+1},\bm{\eta}_{k+1}) =\mathrm{prox}_{f^{*}}(\bm{r}_k + \bm{u}_{k+1},\bm{\eta}_k+2\bm{\sigma}_{k+1}-\bm{\sigma}_k)$\;
		$\bm{\xi}_{k+1}=\mathrm{prox}_{g^{*}}(\bm{\xi}_k+\bm{M}(2\bm{\sigma}_{k+1}-\bm{\sigma}_{k}))$\;
		$\bm{\zeta}_{k+1}=\mathrm{prox}_{\gamma_4 g^{*}}(\bm{\zeta}_k + \gamma_4\bm{M}(2\bm{\tau}_{k+1}-\bm{\tau}_{k}))$\;
		}
\end{algorithm}
\begin{proof}
First, we derive the relation \eqref{eq:Characterization_SolutionSet_AveragedNonexpansiveOp}.
Under the condition \eqref{eq:OverallConvexCondition},
applying Fact \ref{fact:SubdifferentialCalculus}(b) 
to $\partial (J+\iota_{\mathcal{C}})(\bm{x})$ with the expression \eqref{eq:ExpressCostFunc_SumOfConvex} repeatedly by
the differentiability of the first three terms on the right-hand side of \eqref{eq:ExpressCostFunc_SumOfConvex} on $\mathbb{R}^{n}$,
$\mathrm{dom}(\psi \circ \bm{L}) = \mathbb{R}^{n}$ by Lemma \ref{lemm:psi_lsc_convex_fullDomain}, and Lemma \ref{lemm:Qualification_conj_psi_squaredBNorm}(i), we obtain
\begin{align*}
\partial (J+\iota_{\mathcal{C}})(\bm{x}) = \bm{Q}\bm{x}-\bm{A}^{\top}\bm{y}
+\lambda\partial(\psi \circ \bm{L})(\bm{x}) +\partial\iota_{\mathcal{C}}(\bm{x}) 
+\lambda\partial\left(\left(\psi + \frac{1}{2}\|\bm{B}\cdot \|^2\right)^{*}\circ\bm{B}^{\top}\bm{B}\bm{L} \right)(\bm{x}).
\end{align*}
Furthermore, Lemma \ref{lemm:Qualification_conj_psi_squaredBNorm}(ii) and Fact \ref{fact:SubdifferentialCalculus}(c)
yield
\begin{align*}
\partial\left(\left(\psi + \frac{1}{2}\|\bm{B}\cdot \|^2\right)^{*}
\circ\bm{B}^{\top}\bm{B}\bm{L} \right)(\bm{x})
= \bm{L}^{\top}\bm{B}^{\top}\bm{B}\left(\partial\left(\psi + \frac{1}{2}\|\bm{B}\cdot \|^2\right)^{*}(\bm{B}^{\top}\bm{B}\bm{L}\bm{x})\right).
\end{align*}
In addition, by the property \eqref{eq:relation_subdifferential_conjugate}, we obtain
\begin{align*}
\bm{v} \in \partial\left(\psi + \frac{1}{2}\|\bm{B}\cdot \|^2\right)^{*}(\bm{B}^{\top}\bm{B}\bm{L}\bm{x})
&\Leftrightarrow
\bm{B}^{\top}\bm{B}\bm{L}\bm{x} \in
\partial\left(\psi + \frac{1}{2}\|\bm{B}\cdot \|^2\right)(\bm{v})\\
&\Leftrightarrow
\bm{B}^{\top}\bm{B}\bm{L}\bm{x} \in
\partial\psi(\bm{v}) + \bm{B}^{\top}\bm{B}\bm{v},
\end{align*}
where the last equivalence holds since $(1/2)\|\bm{B}\cdot \|^2$ is differentiable on $\mathbb{R}^{m}$.
Thus, we have
\begin{align*}
\bm{x}^{\star} \in \mathcal{S}
&\Leftrightarrow \bm{0}\in \partial (J+\iota_{\mathcal{C}})(\bm{x}^{\star})\\
&\Leftrightarrow 
\left\{\begin{aligned}
&\bm{0} \in \bm{Q}\bm{x}^{\star}-\bm{A}^{\top}\bm{y}+\lambda\partial\left(\psi \circ \bm{L}\right)(\bm{x}^{\star}) +\lambda\bm{L}^{\top}\bm{B}^{\top}\bm{B}\bm{v}^{\star} + \bm{t}^{\star}\\
&\bm{B}^{\top}\bm{B}\bm{L}\bm{x}^{\star} \in
\partial\psi(\bm{v}^{\star}) + \bm{B}^{\top}\bm{B}\bm{v}^{\star}\\
&\bm{t}^{\star} \in \partial\iota_{\mathcal{C}}(\bm{x}^{\star}).
\end{aligned}\right.
\end{align*}
Since Lemma \ref{lemm:nonempty_partialmin_phi} guarantees $\argmin_{\bm{\sigma}\in\mathbb{R}^{l}}\varphi(\bm{L}\bm{x}^{\star},\bm{\sigma}) \neq \varnothing$
and $\argmin_{\bm{\sigma}\in\mathbb{R}^{l}}\varphi(\bm{v}^{\star},\bm{\sigma}) \neq \varnothing$,
we can introduce auxiliary variables $\bm{\sigma}^{\star}$ and $\bm{\tau}^{\star}$ as follows:
\begin{align*}
\bm{x}^{\star} \in \mathcal{S}
&\Leftrightarrow 
\left\{\begin{aligned}
&-\bm{Q}\bm{x}^{\star}+\bm{A}^{\top}\bm{y}-\lambda\bm{L}^{\top}\bm{B}^{\top}\bm{B}\bm{v}^{\star} -\bm{t}^{\star} \in \lambda\partial\left(\psi \circ \bm{L}\right)(\bm{x}^{\star})\\
&\bm{\sigma}^{\star} \in \argmin_{\bm{\sigma}\in\mathbb{R}^{l}}\varphi(\bm{L}\bm{x}^{\star},\bm{\sigma})\\
&\bm{B}^{\top}\bm{B}\bm{L}\bm{x}^{\star}-\bm{B}^{\top}\bm{B}\bm{v}^{\star} \in 
\partial\psi(\bm{v}^{\star})\\
&\bm{\tau}^{\star} \in \argmin_{\bm{\sigma}\in\mathbb{R}^{l}}\varphi(\bm{v}^{\star},\bm{\sigma})\\
&\bm{t}^{\star} \in \partial\iota_{\mathcal{C}}(\bm{x}^{\star})
\end{aligned}\right.\\
&\Leftrightarrow 
\left\{\begin{aligned}
&(-\bm{Q}\bm{x}^{\star}+\bm{A}^{\top}\bm{y}-\lambda\bm{L}^{\top}\bm{B}^{\top}\bm{B}\bm{v}^{\star}-\bm{t}^{\star},\bm{0})
\in \lambda\partial \left(\varphi\circ\bar{\bm{L}}\right)(\bm{x}^{\star},\bm{\sigma}^{\star})\\
&(\bm{B}^{\top}\bm{B}\bm{L}\bm{x}^{\star}-\bm{B}^{\top}\bm{B}\bm{v}^{\star},\bm{0}) \in 
\partial\varphi(\bm{v}^{\star},\bm{\tau}^{\star})\\
&\bm{t}^{\star} \in \partial\iota_{\mathcal{C}}(\bm{x}^{\star}),
\end{aligned}\right.
\end{align*}
where the last equivalence follows from Lemma \ref{lemma:subdifferential_MarginalFunc}.
Applying Fact \ref{fact:SubdifferentialCalculus}(c) under Assumption \ref{Assumption:QualificationConditions}(ii) 
and Fact \ref{fact:SubdifferentialCalculus}(a) under Assumption \ref{Assumption:QualificationConditions}(iii)
with the definitions of $\varphi$, $\bar{\bm{L}}$,
and $\bar{\bm{M}}$ in \eqref{eq:represent_seed_function}, \eqref{eq:def:L_bar}, and \eqref{eq:def:M_bar},
we obtain
\begin{align*}
\partial \left(\varphi\circ\bar{\bm{L}}\right)(\bm{x}^{\star},\bm{\sigma}^{\star}) &= 
\bar{\bm{L}}^{\top}[\partial\varphi (\bm{L}\bm{x}^{\star},\bm{\sigma}^{\star})]\\
&=\bar{\bm{L}}^{\top}\left[\partial\left(f + g\circ \bar{\bm{M}} \right) (\bm{L}\bm{x}^{\star},\bm{\sigma}^{\star})\right]\\
&=\bar{\bm{L}}^{\top}\left[\partial f(\bm{L}\bm{x}^{\star},\bm{\sigma}^{\star}) +\bar{\bm{M}}^{\top}(\partial g(\bm{M}\bm{\sigma}^{\star})) \right],
\end{align*}
and similarly,
\begin{align*}
\partial\varphi(\bm{v}^{\star},\bm{\tau}^{\star}) = \partial f(\bm{v}^{\star},\bm{\tau}^{\star}) +\bar{\bm{M}}^{\top}(\partial g(\bm{M}\bm{\tau}^{\star})).
\end{align*}
Thus, using the relations
\begin{align*}
(\bm{r}^{\star},\bm{\eta}^{\star}) \in \partial f(\bm{L}\bm{x}^{\star},\bm{\sigma}^{\star}) &\Leftrightarrow
(\bm{L}\bm{x}^{\star},\bm{\sigma}^{\star}) \in \partial f^{*}(\bm{r}^{\star},\bm{\eta}^{\star}),\\
\bm{\xi}^{\star} \in \partial g(\bm{M}\bm{\sigma}^{\star})&\Leftrightarrow
\bm{M}\bm{\sigma}^{\star} \in \partial g^{*}(\bm{\xi}^{\star}),\\
\bm{\zeta}^{\star} \in \partial g(\bm{M}\bm{\tau}^{\star})&\Leftrightarrow
\bm{M}\bm{\tau}^{\star} \in \partial g^{*}(\bm{\zeta}^{\star}),
\end{align*}
since $\bar{\bm{L}}^{\top}[(\bm{r}^{\star},\bm{\eta}^{\star})+\bar{\bm{M}}^{\top}\bm{\xi}^{\star}] = (\bm{L}^{\top}\bm{r}^{\star},\bm{\eta}^{\star}+\bm{M}^{\top}\bm{\xi}^{\star})$
and $\bar{\bm{M}}^{\top}\bm{\zeta}^{\star} = (\bm{0},\bm{M}^{\top}\bm{\zeta}^{\star})$,
we deduce
\begin{align*}
\bm{x}^{\star} \in \mathcal{S}
\Leftrightarrow 
\left\{\begin{aligned}
&\bm{0} \in \bm{Q}\bm{x}^{\star} -\bm{A}^{\top}\bm{y} +\lambda\bm{L}^{\top}\bm{B}^{\top}\bm{B}\bm{v}^{\star} + \lambda\bm{L}^{\top}\bm{r}^{\star} +\partial\iota_{\mathcal{C}}(\bm{x}^{\star})\\
&\bm{0} =\lambda\bm{\eta}^{\star}+ \lambda\bm{M}^{\top}\bm{\xi}^{\star}\\
&(\bm{0},\bm{0}) \in 
\lambda\partial f(\bm{v}^{\star},\bm{\tau}^{\star})+\lambda(\bm{B}^{\top}\bm{B}(
\bm{v}^{\star}-\bm{L}\bm{x}^{\star}),\bm{M}^{\top}\bm{\zeta}^{\star})\\
&(\bm{0},\bm{0}) \in \lambda\partial f^{*}(\bm{r}^{\star},\bm{\eta}^{\star}) - \lambda(\bm{L}\bm{x}^{\star},\bm{\sigma}^{\star})\\
&\bm{0} \in \lambda\partial g^{*}(\bm{\xi}^{\star})-\lambda\bm{M}\bm{\sigma}^{\star}\\
&\bm{0} \in \lambda\partial g^{*}(\bm{\zeta}^{\star}) - \lambda\bm{M}\bm{\tau}^{\star}.
\end{aligned}\right.
\end{align*}
By using the notation
$\bm{z} = (\bm{x},
\bm{\sigma},
\bm{v},
\bm{\tau},
\bm{r},
\bm{\eta},
\bm{\xi},
\bm{\zeta})$,
the last expression can be rewritten as
\begin{align}
\label{eq:CharacterizeSolution_MonotoneInclusion}
\bm{x}^{\star} \in \mathcal{S} &\Leftrightarrow \bm{0} \in H(\bm{z}^{\star}) + \mathcal{D}(\bm{z}^{\star}) + \bm{N}\bm{z}^{\star}
\Leftrightarrow
\bm{P}\bm{z}^{\star} - H(\bm{z}^{\star}) \in \bm{P}\bm{z}^{\star}+ \mathcal{D}(\bm{z}^{\star}) + \bm{N}\bm{z}^{\star},
\end{align}
where $H\colon\mathcal{H}\rightarrow\mathcal{H}$ is the affine operator, $\mathcal{D}\colon\mathcal{H}\rightarrow2^{\mathcal{H}}$ is the set-valued operator consisting of subdifferentials,
and $\bm{N}\in\mathbb{R}^{\mathrm{dim}\mathcal{H}\times\mathrm{dim}\mathcal{H}}$ is the skew-symmetric matrix, respectively defined by
\begin{align}
\label{eq:def:AffineOp_H}
H(\bm{z}) &\coloneqq  (\bm{Q}\bm{x} -\bm{A}^{\top}\bm{y},\bm{0},\lambda\bm{B}^{\top}\bm{B}\bm{v},\bm{0},\bm{0},\bm{0},\bm{0},\bm{0}),\\
\label{eq:def:subDiffOp_D}
\mathcal{D}(\bm{z})
&\coloneqq  \partial\iota_{\mathcal{C}}(\bm{x})\times\{\bm{0}\}\times 
\lambda\partial f(\bm{v},\bm{\tau})
\times \lambda\partial f^{*}(\bm{r},\bm{\eta})\times
\lambda\partial g^{*}(\bm{\xi})\times
\lambda\partial g^{*}(\bm{\zeta}),\\
\label{eq:def:skewOp_N}
\bm{N} &\coloneqq  \bm{\Pi}^{\top}\begin{bmatrix}
\bm{N}_1 &  &\\
 & \bm{N}_2 & \\
 &  & \bm{N}_3
\end{bmatrix}\bm{\Pi},
\end{align}
where $\bm{\Pi}\in\mathbb{R}^{\mathrm{dim}\mathcal{H}\times\mathrm{dim}\mathcal{H}}$ is the permutation matrix defined as
\begin{align}
\label{eq:def:permutationOp}
\bm{\Pi}\colon(\bm{x},
\bm{\sigma},
\bm{v},
\bm{\tau},
\bm{r},
\bm{\eta},
\bm{\xi},
\bm{\zeta})
\mapsto(\bm{x},
\bm{v},
\bm{r},
\bm{\sigma},
\bm{\eta},
\bm{\xi},
\bm{\tau},
\bm{\zeta}),
\end{align}
and
\begin{align*}
\bm{N}_{1}&\coloneqq \lambda\begin{bmatrix}
\bm{O} & \bm{L}^{\top}\bm{B}^{\top}\bm{B} & \bm{L}^{\top}\\
- \bm{B}^{\top}\bm{B}\bm{L} & \bm{O} & \bm{O}\\
-\bm{L} & \bm{O} & \bm{O}
\end{bmatrix},\\
\bm{N}_{2}&\coloneqq \lambda\begin{bmatrix}
\bm{O} & \bm{I} & \bm{M}^{\top}\\
-\bm{I} & \bm{O} & \bm{O}\\
- \bm{M} & \bm{O} & \bm{O}
\end{bmatrix},\\
\bm{N}_{3}&\coloneqq \lambda\begin{bmatrix}
\bm{O} &  \bm{M}^{\top}\\
- \bm{M} & \bm{O}
\end{bmatrix}.
\end{align*}
Meanwhile,
since the proximity operator is the resolvent of subdifferential (see \eqref{eq:relationProxSubdifferential}),
we derive from \eqref{eq:def:AveragedOperatorForGME-MI} that
\begin{align}
\nonumber
&T_{\text{\rm GME-MI}}(\bm{z}) = \hat{\bm{z}}\\
\nonumber
\Leftrightarrow &
\left\{\begin{aligned}
&
\bm{x}-\gamma_1(\bm{Q}\bm{x} -\bm{A}^{\top}\bm{y}+\lambda\bm{L}^{\top}\bm{B}^{\top}\bm{B}\bm{v}+\lambda\bm{L}^{\top}\bm{r})\in (\mathrm{Id}+\partial \iota_{\mathcal{C}})(\hat{\bm{x}})
\\
&\bm{\sigma}-\gamma_2(\bm{\eta}+\bm{M}^{\top}\bm{\xi}) = \hat{\bm{\sigma}}\\
&
(\bm{v}+\gamma_3\bm{B}^{\top}\bm{B}(\bm{L}(2\hat{\bm{x}}-\bm{x})-\bm{v}),
\bm{\tau}-\gamma_3\bm{M}^{\top}\bm{\zeta})
\in(\mathrm{Id}+\gamma_3\partial f)(\hat{\bm{v}},\hat{\bm{\tau}})\\
&
(\bm{r} + \bm{L}(2\hat{\bm{x}}-\bm{x}),
\bm{\eta}+2\hat{\bm{\sigma}}-\bm{\sigma})\in
(\mathrm{Id}+\partial f^{*})(\hat{\bm{r}},\hat{\bm{\eta}})\\
&\bm{\xi}+\bm{M}(2\hat{\bm{\sigma}}-\bm{\sigma})\in(\mathrm{Id}+\partial g^{*})(\hat{\bm{\xi}})\\
&\bm{\zeta} + \gamma_4\bm{M}(2\hat{\bm{\tau}}-\bm{\tau})\in (\mathrm{Id}+\gamma_4\partial g^{*})(\hat{\bm{\zeta}})
\end{aligned}\right.\\
\nonumber\Leftrightarrow&
\left\{\begin{aligned}
&\frac{1}{\gamma_1}\bm{x}-\lambda\bm{L}^{\top}\bm{B}^{\top}\bm{B}\bm{v}-\lambda\bm{L}^{\top}\bm{r}-\bm{Q}\bm{x} + \bm{A}^{\top}\bm{y} = \frac{1}{\gamma_1}\hat{\bm{x}}+\partial \iota_{\mathcal{C}}(\hat{\bm{x}})\\
&\frac{\lambda}{\gamma_2}\bm{\sigma}-\lambda\bm{\eta}- \lambda\bm{M}^{\top}\bm{\xi} =
\frac{\lambda}{\gamma_2}\hat{\bm{\sigma}}\\
&\left(\frac{\lambda}{\gamma_3}\bm{v}-\lambda\bm{B}^{\top}\bm{B}\bm{L}\bm{x}-\lambda \bm{B}^{\top}\bm{B}\bm{v},
\frac{\lambda}{\gamma_3}\bm{\tau}-\lambda\bm{M}^{\top}\bm{\zeta}\right) \in
\frac{\lambda}{\gamma_3}(\hat{\bm{v}},\hat{\bm{\tau}})+
\lambda\partial f(\hat{\bm{v}},\hat{\bm{\tau}})-
(2\lambda\bm{B}^{\top}\bm{B}\bm{L}\hat{\bm{x}},\bm{0})\\
&
(\lambda \bm{r}-\lambda \bm{L}\bm{x},
\lambda\bm{\eta}-\lambda\bm{\sigma})\in
\lambda(\hat{\bm{r}},
\hat{\bm{\eta}})
+ \lambda\partial f^{*}(
\hat{\bm{r}},
\hat{\bm{\eta}}) - 
(2\lambda\bm{L}\hat{\bm{x}},
2\lambda\hat{\bm{\sigma}})\\
&\lambda \bm{\xi}-\lambda\bm{M}\bm{\sigma} \in
\lambda \hat{\bm{\xi}}+\lambda\partial g^{*}(\hat{\bm{\xi}})-2\lambda\bm{M}\hat{\bm{\sigma}}\\
&\frac{\lambda}{\gamma_4}\bm{\zeta}- \lambda\bm{M}\bm{\tau} \in 
\frac{\lambda}{\gamma_4}\hat{\bm{\zeta}} +\lambda\partial g^{*}(\hat{\bm{\zeta}}) - 2\lambda\bm{M}\hat{\bm{\tau}}
\end{aligned}\right.\\
\label{eq:GME_MI_Op_Representation}
\Leftrightarrow{}&\bm{P}\bm{z} - H(\bm{z}) \in  \bm{P}\hat{\bm{z}} + \mathcal{D}(\hat{\bm{z}})+\bm{N}\hat{\bm{z}}.
\end{align}
This relation for $\bm{z} = \hat{\bm{z}} = \bm{z}^{\star}$ and \eqref{eq:CharacterizeSolution_MonotoneInclusion} yield the claim \eqref{eq:Characterization_SolutionSet_AveragedNonexpansiveOp}.

Next, we prove that $T_{\text{\rm GME-MI}}$ is $\kappa/(2\kappa-1)$-averaged nonexpansive in $(\mathcal{H},\langle\cdot,\cdot\rangle_{\bm{P}},\|\cdot\|_{\bm{P}})$.
The positive definiteness of $\bm{P}$ is shown in \ref{proof:positive_definite_P}.
Since $\mathcal{D}$ in \eqref{eq:def:subDiffOp_D}
is the product of subdifferentials, it is maximally monotone
in $(\mathcal{H},\langle\cdot,\cdot\rangle,\|\cdot\|)$
by \cite[Theorem 20.25 and Proposition 20.23]{BC:ConvexAnalysis}.
Since $\bm{N}$ in \eqref{eq:def:skewOp_N} is
skew-symmetric,
it is maximally monotone
in $(\mathcal{H},\langle\cdot,\cdot\rangle,\|\cdot\|)$ by \cite[Example 20.35]{BC:ConvexAnalysis}.
Thus, since the domain of $\bm{N}$ is $\mathcal{H}$,
$\mathcal{D}+\bm{N}$ is maximally monotone in $(\mathcal{H},\langle\cdot,\cdot\rangle,\|\cdot\|)$
by \cite[Corollary 25.5]{BC:ConvexAnalysis}.
The positive definiteness of $\bm{P}$ and the maximal monotonicity of $\mathcal{D}+\bm{N}$
imply that $\bm{P}^{-1}\circ (\mathcal{D} + \bm{N})$
is maximally monotone in $(\mathcal{H},\langle\cdot,\cdot\rangle_{\bm{P}},\|\cdot\|_{\bm{P}})$
by \cite[Proposition 20.24]{BC:ConvexAnalysis}. Thus,
its resolvent
$(\mathrm{Id}+\bm{P}^{-1}\circ (\mathcal{D} + \bm{N}))^{-1}$ is single-valued
and $1/2$-averaged nonexpansive in $(\mathcal{H},\langle\cdot,\cdot\rangle_{\bm{P}},\|\cdot\|_{\bm{P}})$.
Since $(\mathrm{Id}+\bm{P}^{-1}\circ (\mathcal{D} + \bm{N}))^{-1}$ is single-valued, we 
derive from \eqref{eq:GME_MI_Op_Representation} that
\begin{align*}
T_{\text{\rm GME-MI}}(\bm{z}) = \hat{\bm{z}}
&\Leftrightarrow \bm{P}\bm{z} - H(\bm{z}) \in \bm{P}\hat{\bm{z}}+ \mathcal{D}(\hat{\bm{z}}) + \bm{N}\hat{\bm{z}}\\
&\Leftrightarrow 
\bm{z}-\bm{P}^{-1} H(\bm{z}) \in \hat{\bm{z}} + \bm{P}^{-1} (\mathcal{D} + \bm{N})(\hat{\bm{z}})\\
&\Leftrightarrow 
 [(\mathrm{Id}+\bm{P}^{-1}\circ (\mathcal{D} + \bm{N}))^{-1}\circ(\mathrm{Id}-\bm{P}^{-1}\circ H)](\bm{z})
 =\hat{\bm{z}},
\end{align*}
i.e.,
\begin{align*}
T_{\text{\rm GME-MI}}
= (\mathrm{Id}+\bm{P}^{-1} \circ (\mathcal{D} + \bm{N}))^{-1} \circ (\mathrm{Id}-\bm{P}^{-1} \circ H).
\end{align*}
In $(\mathcal{H},\langle\cdot,\cdot\rangle_{\bm{P}},\|\cdot\|_{\bm{P}})$,
the $1/2$-averaged nonexpansiveness of
$(\mathrm{Id}+\bm{P}^{-1} \circ (\mathcal{D} + \bm{N}))^{-1}$
and the $1/\kappa$-averaged nonexpansiveness of
$\mathrm{Id}-\bm{P}^{-1} \circ H$, proved in \ref{proof:AveragedNonexpansive_AffineOp_H},
yield that
$T_{\text{\rm GME-MI}}$ is $\kappa/(2\kappa-1)$-averaged nonexpansive by Fact \ref{fact:CompositionAveragedOp}.
Finally,
since Algorithm \ref{alg:Optimization_GME_MI} is designed
to have $\bm{z}_{k+1} = T_{\text{\rm GME-MI}}(\bm{z}_k)$,
the averaged nonexpansiveness
of $T_{\text{\rm GME-MI}}$
and the relation \eqref{eq:Characterization_SolutionSet_AveragedNonexpansiveOp} yield that
$(\bm{x}_k)_{k\in\mathbb{N}}$
converges to a point in $\mathcal{S}$ by Fact \ref{fact:KM_Iteration}.
\end{proof}

\begin{rmrk}[Computational cost of Algorithm \ref{alg:Optimization_GME_MI}]
Algorithm \ref{alg:Optimization_GME_MI} consists of
matrix-vector multiplication, the proximity operators, and the projection onto the constraint set.
Matrix-vector multiplication admits efficient implementation in general.
This paper focuses on simple constraints, e.g., the box constraint for the dynamic range of signal,
so that the projection onto the constraint set can be easily computed.
While computational cost of the proximity operators
depends on the specific scenario,
the proximity operators needed for
the GME-LOP-$\ell_2/\ell_1$ and GME-TGV models
have linear computational complexity (see \ref{appendix:ComputeProxOp}).
Furthermore, as seen from \ref{appendix:ComputeProxOp},
these proximity operators can be computed component-wise,
and thus Algorithm \ref{alg:Optimization_GME_MI} is inherently suitable for parallel computation,
which could reduce the computation time.
\end{rmrk}

\begin{rmrk}[Flexibility of proposed framework]\hfill
\begin{enumerate}
\item[a)] (Combination with various observation models) Since $\bm{A}$ in \eqref{eq:def:OverallRegularizationModel} can be
any linear matrix, the proposed framework is applicable to various linear inverse problems, e.g.,
denoising, deconvolution, compressed sensing, and remote sensing, among many others.

\item[b)] (Enhancement of various MI penalties)
While we focus on enhancement of the basic forms of
the LOP-$\ell_2/\ell_1$ and TGV penalties in Example \ref{exmp:GME_MI_penalties} to simplify the exposition,
the proposed framework is potentially applicable to enhancement of their variants
for diverse scenarios, e.g.,
\cite{Kuroda:GraphSparse}
for graph-structured sparse signals, \cite{Kuroda:PSDEstimation} for power spectral densities,
\cite{Arai:TF} for time-frequency representations,
\cite{Ferstl:TGV} for depth images,
\cite{Bredies:TGV_BookChapter,Ono:VTGV} for multi-channel images,
\cite{Ranftl:NonLocalTGV} for signals with nonlocal similarities,
and \cite{Gao:OTGV} for textures in images.
\end{enumerate}
\end{rmrk}

\section{Numerical Examples}
\label{sect:Experiment}
To demonstrate the effectiveness of the proposed framework to enhance MI penalties,
we present numerical examples
of the GME-MI models shown in Example \ref{exmp:GME_MI_penalties} for enhancement of the LOP-$\ell_{2}/\ell_{1}$ and TGV penalties.

\subsection{Estimation of Block-Sparse Signals}
\label{sect:Experiment_blocksparse}
We conduct numerical experiments corresponding to the scenario in Examples \ref{exmp:MI_penalties}--\ref{exmp:GME_MI_penalties}(a)
for block-sparse estimation with unknown block partition.
Note that we adopt this scenario because the LOP-$\ell_{2}/\ell_{1}$ penalty is proposed
to exploit block-sparsity without the knowledge of concrete block partition.
More precisely, we consider the estimation of block-sparse signal $\bm{x}_{\text{\rm org}} \in \mathcal{C}\coloneqq  \mathbb{R}^{n}$
from noisy compressive measurements $\bm{y} = \bm{A}\bm{x}_{\text{\rm org}} + \bm{\varepsilon} \in \mathbb{R}^{d}$,
where the entries of $\bm{A} \in \mathbb{R}^{d\times n}$
are drawn from i.i.d.~Gaussian distribution $\mathcal{N}(0,1)$, 
$\bm{\varepsilon} \in \mathbb{R}^{d}$ is the white Gaussian noise, and $d < n$.
The block-sparse signal $\bm{x}_{\text{\rm org}}$
with $4$ nonzero blocks, $80$ nonzero components, and $n = 256$
is randomly generated by the scheme used in \cite{Kuroda:BlockSparse}.
Note that the block partition is randomly generated for each trial to investigate
the average performance for various block partitions.
Amplitudes of nonzero components are drawn from i.i.d.~$\mathcal{N}(0,1)$.

As an instance of the GME-MI model,
the GME-LOP-$\ell_2/\ell_1$ penalty $\Psi_{\bm{B}, \alpha}^{\text{\rm LOP}}$
presented in Example \ref{exmp:GME_MI_penalties}(a) is used in Definition \ref{defn:GMEMI_Regularization_Model}.
Since $\bm{x}_{\text{\rm org}}$ itself is block-sparse in the experiments, we set $\bm{L} = \bm{I}$.
Due to $\bm{L} = \bm{I}$,
we can set $\bm{B}$ satisfying \eqref{eq:OverallConvexCondition} by
$\bm{B} = \sqrt{\theta/\lambda}\bm{A}$ with any $\theta \in [0,1]$.

\begin{figure}[t]
\centering
\subfloat[\scriptsize  NMSE versus the number of measurements, where the SNR is fixed to $40$dB.]{\includegraphics[width=0.4\columnwidth]{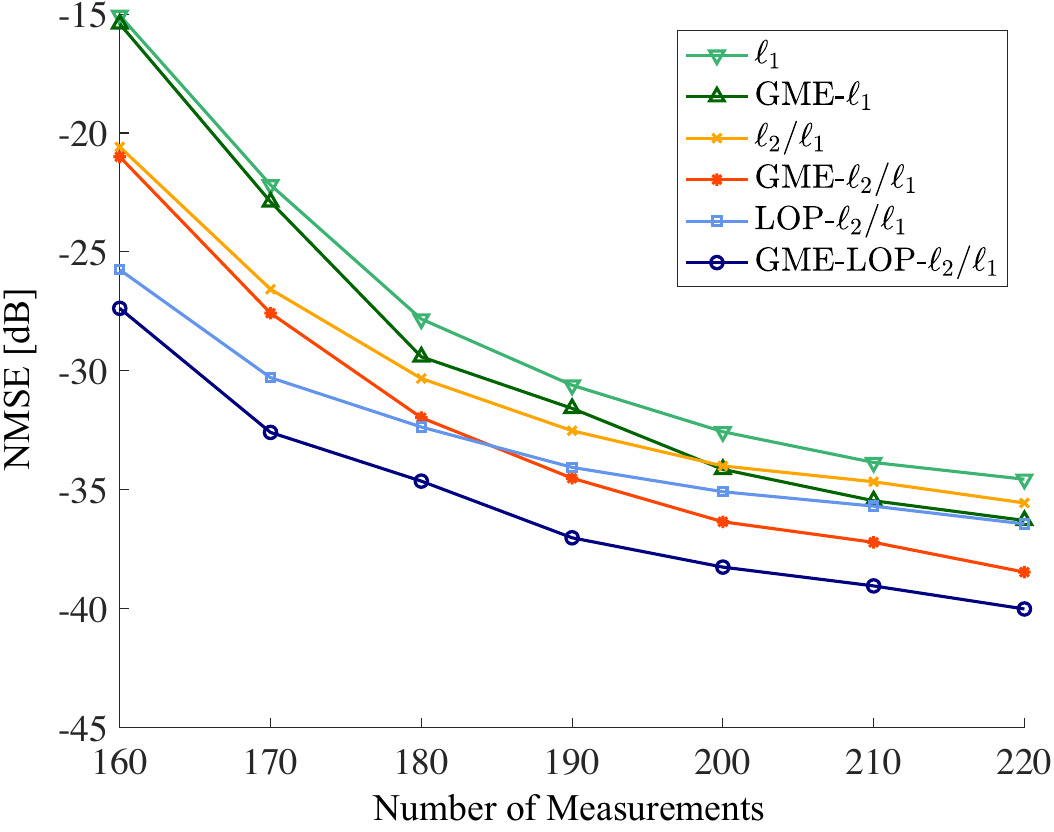}}
\hspace{20pt}
\centering
\subfloat[\scriptsize  NMSE versus the SNR, where the number of measurements is fixed to $220$.]{\includegraphics[width=0.4\columnwidth]{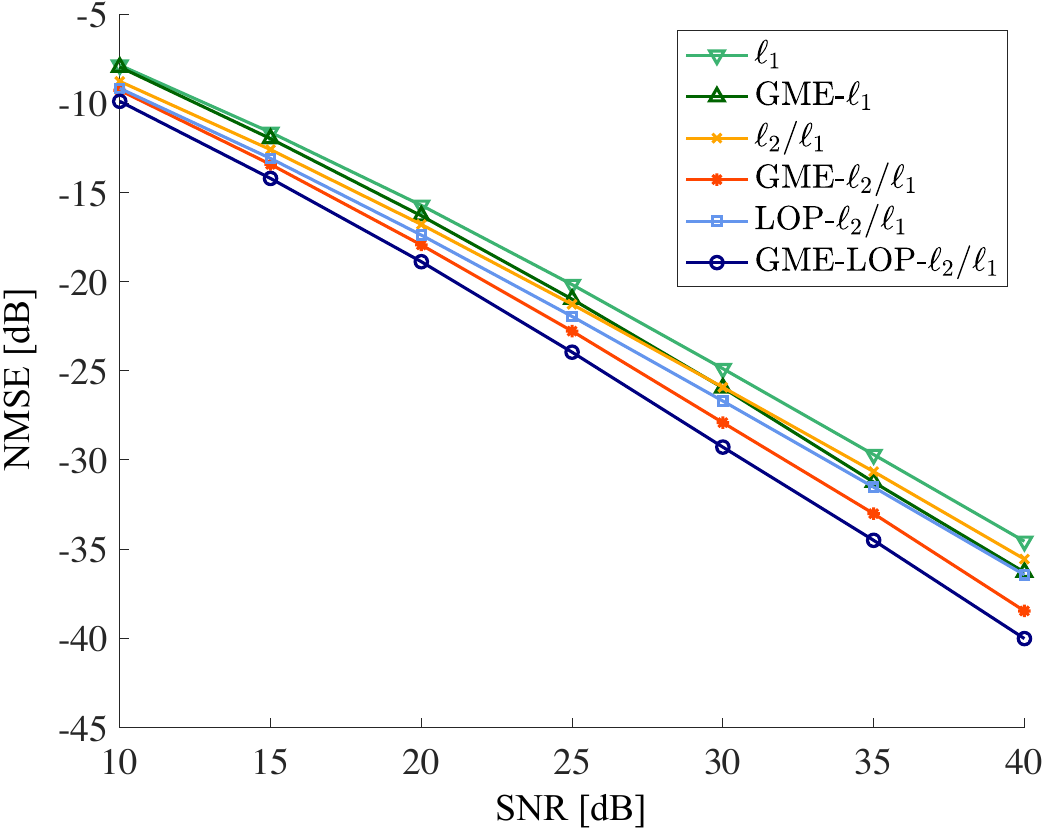}}
\caption{Average NMSEs for estimation of block-sparse signals.}
\label{fig:NMSE_BlockSparse}
\end{figure}
\begin{figure}[t]
  \centering
    \includegraphics[width=0.65\columnwidth]{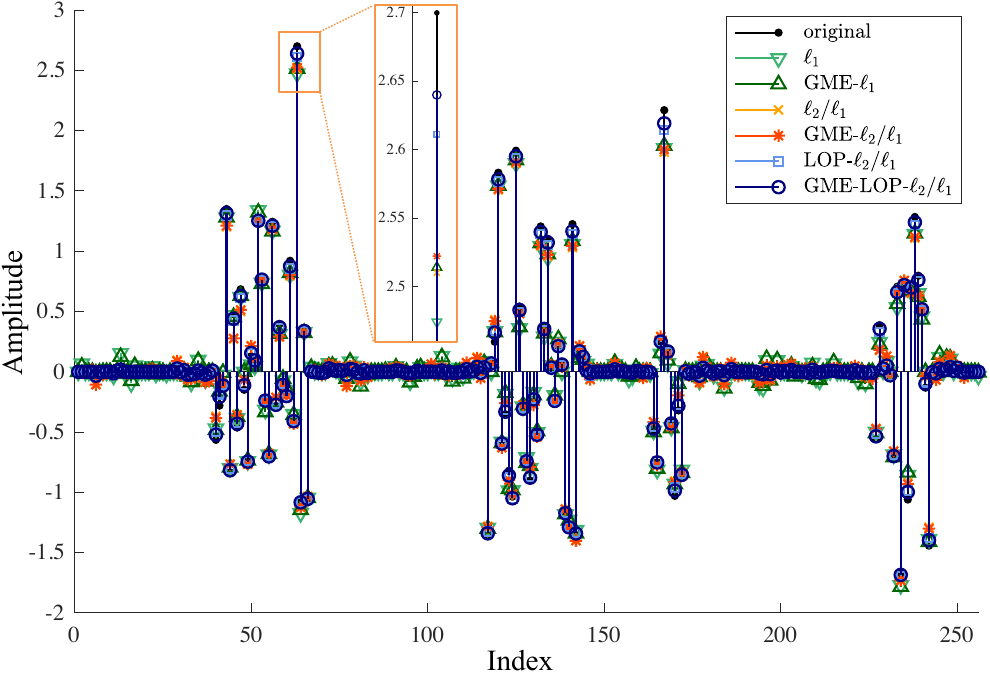}
  \caption{Original block-sparse signal and the estimates for a trial, where the number of measurements is $160$ and the SNR is $40$dB.}
  \label{fig:trial_experiment_BlockSparse}
\end{figure}

We compare the proposed GME-LOP-$\ell_{2}/\ell_{1}$ penalty with
the LOP-$\ell_{2}/\ell_{1}$ penalty given in Example \ref{exmp:MI_penalties}(a)
to investigate the effectiveness of enhancement by the proposed framework.
For reference, we also show
the GME penalty $R_{\bm{B}}(\bm{x}) \coloneqq  R(\bm{x}) - \min_{\bm{v} \in \mathbb{R}^{n}}[R(\bm{v}) + (1/2)\|\bm{B}(\bm{x}-\bm{v})\|^2]$ \cite{Abe:LiGME}
for convex prox-friendly penalties:
$R = \|\cdot\|_{2,1}$ for block-sparsity with fixed blocks
(proposed in \cite{Kitahara:LiGME_MRI,Liu:GroupGMC})
and $R = \|\cdot\|_{1}$ for non-structured sparsity (proposed as GMC penalty \cite{Selesnick:GMC}).
Pre-enhanced $\ell_{2}/\ell_{1}$ and $\ell_1$ penalties
are also shown for reference.
All the penalties are combined with the square error $(1/2)\|\bm{y}-\bm{A}\bm{x}\|^2$.

The GME-LOP-$\ell_{2}/\ell_{1}$ model is solved by Algorithm \ref{alg:Optimization_GME_MI}
with the representation given in Example \ref{exmp:representation_MIpenalty_forOptimization}(a),
the GME models for $R = \|\cdot\|_{2,1}$ and $R = \|\cdot\|_{1}$ are solved by the iterative algorithm
given in \cite{Abe:LiGME},
and the convexly regularized models are solved by the iterative algorithm given in \cite{Kuroda:BlockSparse}.
We terminate the iteration when the norm of the differences between the variables of successive iterates is below the threshold $10^{-4}$ or the number of iterations reaches $10000$.
For the GME-LOP-$\ell_{2}/\ell_{1}$ model,
the computational complexity of Algorithm \ref{alg:Optimization_GME_MI} per iteration
is $\mathcal{O}(dn)$ for matrix-vector multiplication
and $\mathcal{O}(n)$ for the proximity operators (see \ref{appendix:ComputeProxOp}).
This complexity is identical to those of the other algorithms.

We compare the models
in terms of normalized mean square error (NMSE)
\begin{align*}
\frac{\|\bm{x}_{\text{\rm org}}-\bm{x}^{\star}\|^2}{\|\bm{x}_{\text{\rm org}}\|^2},
\end{align*}
where $\bm{x}^{\star}$ is an optimal solution of each regularization model.
Note that the regularization parameter $\lambda$ and other tuning parameters
are adjusted to obtain the best estimation accuracy
for each pair of models and experimental conditions.
The tuning parameters other than $\lambda$ are as follows:
$\alpha$ for GME-LOP-$\ell_{2}/\ell_{1}$ and LOP-$\ell_{2}/\ell_{1}$,
block-size for GME-$\ell_{2}/\ell_{1}$ and $\ell_{2}/\ell_{1}$, and
$\theta$ for the overall convexity conditions of GME-LOP-$\ell_{2}/\ell_{1}$, GME-$\ell_{2}/\ell_{1}$, and GME-$\ell_{1}$.
In Figs.~\ref{fig:NMSE_BlockSparse}(a) and (b), NMSE is plotted
with respect to the number, $d$, of measurements
and the SNR
$E[\|\bm{A}\bm{x}_{\text{\rm org}}\|^2]/E[\|\bm{\varepsilon}\|^2]$, respectively,
where the results are averaged over $100$ independent trials.
As is clear from Fig.~\ref{fig:NMSE_BlockSparse},
the GME-LOP-$\ell_{2}/\ell_{1}$ model
achieves the best estimation accuracy, representing a
significant improvement over the LOP-$\ell_{2}/\ell_{1}$ model.

We also discuss an example involving the original block-sparse signal and the estimates for a trial in
Fig.~\ref{fig:trial_experiment_BlockSparse}.
The GME-LOP-$\ell_{2}/\ell_{1}$ model certainly remedies the underestimation tendency of the LOP-$\ell_{2}/\ell_{1}$ model.
Meanwhile, the $\ell_1$ and $\ell_2/\ell_1$ models
yield multiple incorrect nonzero components, possibly due to their fixed structures
in contrast to the (GME-)LOP-$\ell_{2}/\ell_{1}$ model that automatically optimizes the block structure.
This drawback is not resolved by the GME-$\ell_1$ and GME-$\ell_{2}/\ell_{1}$ models,
while mitigating the underestimation tendency.

\subsection{Estimation of Piecewise Linear Signals}
\label{sect:Experiment_piecewiselinear}
We conduct numerical experiments for the scenario in Examples \ref{exmp:MI_penalties}--\ref{exmp:GME_MI_penalties}(b)
to illustrate the estimation of piecewise linear signals.
Note that we adopt this scenario because
the most suitable model for
the second-order TGV penalty is piecewise linear.
We set $\bm{x}_{\text{\rm org}} \in \mathcal{C}\coloneqq  
[-1, 1]^n$ as the piecewise linear signal indicated by the dotted line in Fig.~\ref{fig:trial_experiment_PiecewiseLinear},
where $n = 128$.
Measurements are generated in the same way as in the previous section.

\begin{figure}[t]
  \centering
    \includegraphics[width=0.9\columnwidth]{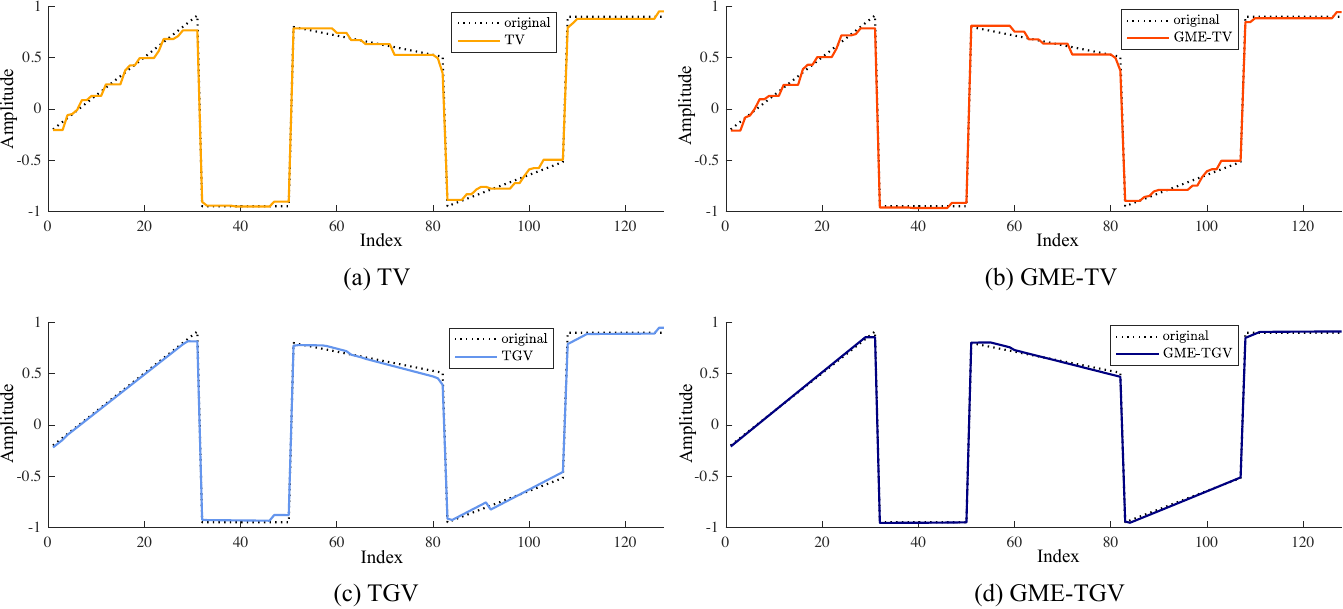}
  \caption{Original piecewise linear signal and the estimates for a trial, where the number of measurements is $100$ and the SNR is $20$dB.}
  \label{fig:trial_experiment_PiecewiseLinear}
\end{figure}
\begin{figure}[t]
\centering
\subfloat[\scriptsize  NMSE versus the number of measurements, where the SNR is fixed to $30$dB.]{\includegraphics[width=0.4\columnwidth]{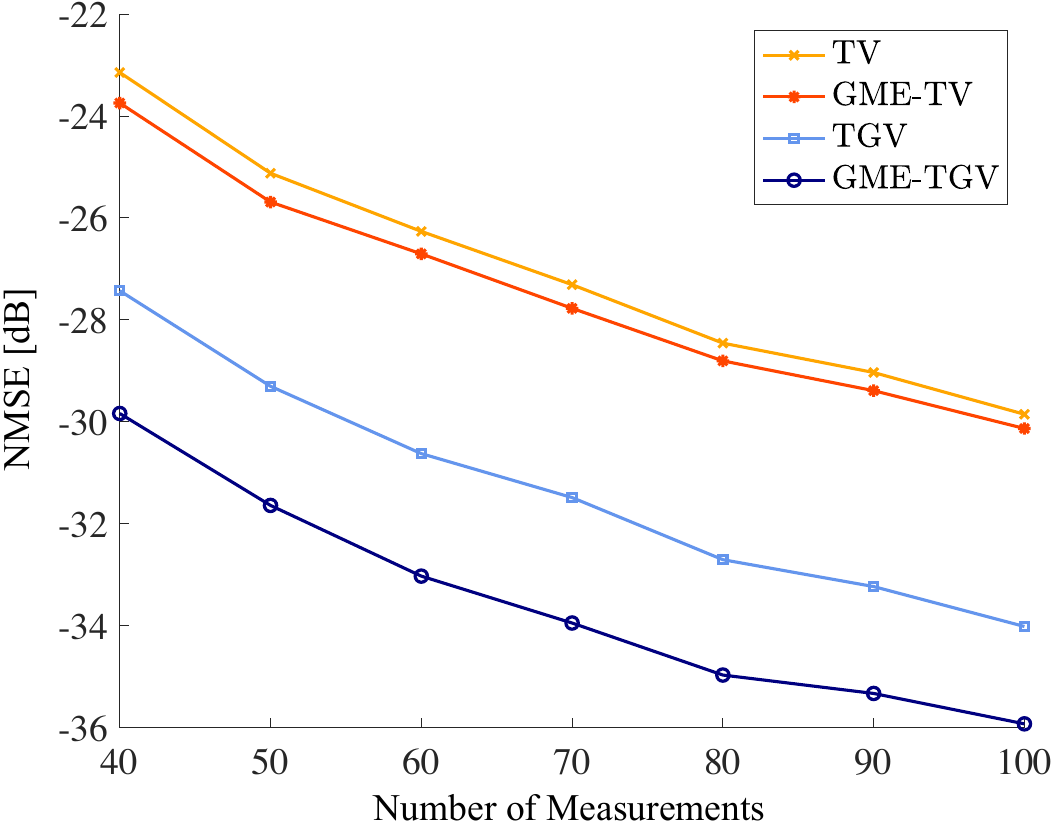}}
\hspace{20pt}
\centering
\subfloat[\scriptsize  NMSE versus the SNR, where the number of measurements is fixed to $100$.]{\includegraphics[width=0.4\columnwidth]{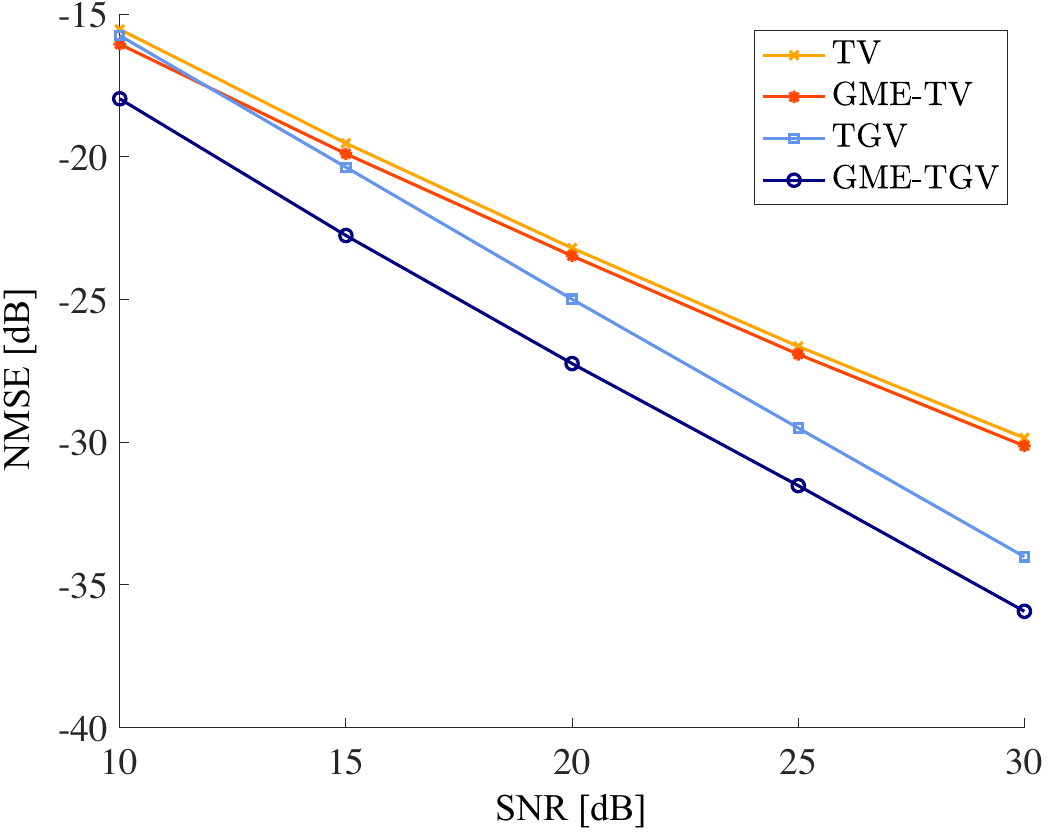}}
\caption{Average NMSEs for estimation of piecewise linear signals.}
\label{fig:NMSE_PiecewiseLinear}
\end{figure}

As an instance of the GME-MI model,
the (second-order) GME-TGV penalty $\Psi_{\bm{B}, \alpha}^{\text{\rm TGV}}$
presented in Example \ref{exmp:GME_MI_penalties}(b) is used in Definition \ref{defn:GMEMI_Regularization_Model},
where we set $\bm{L} = \bm{D}_{\text{\rm 1d}}$ with $\bm{D}_{\text{\rm 1d}}\in\mathbb{R}^{(n-1)\times n}$
defined as $\bm{D}_{\text{\rm 1d}}\colon\bm{x}\mapsto(x_{i+1}-x_{i})_{i=1}^{n-1}$.
In this case, $\tilde{\bm{D}}$ in the seed function $\varphi_{\alpha}^{\text{\rm TGV}}$ in \eqref{eq:def:seed_function_TGV}
becomes $\bm{D}_{\text{\rm 1d}}^{\top}$.
Since only $\bm{B}^{\top}\bm{B}$ is used in Algorithm \ref{alg:Optimization_GME_MI},
we set $\bm{B}^{\top}\bm{B}$ satisfying \eqref{eq:OverallConvexCondition} for $\bm{L} = \bm{D}_{\text{\rm 1d}}$ as follows. We have the equality
\begin{align*}
\underbrace{\begin{bmatrix}
-1&  1 & &  \\
  & \ddots &\ddots &  \\
   & & -1 & 1
\end{bmatrix}}_{=\bm{D}_{\text{\rm 1d}}\in\mathbb{R}^{(n-1) \times n}}
\underbrace{\begin{bmatrix}
1&   & &  \\
1& 1  & &  \\
 \vdots& \vdots &\ddots &  \\
  1&1 & \cdots & 1
\end{bmatrix}}_{\eqqcolon \bm{S}_{\text{\rm 1d}}\in\mathbb{R}^{n \times n}} =
\begin{bmatrix}
\bm{0} & \bm{I}
\end{bmatrix}.
\end{align*}
Thus, based on \cite[Theorem 1]{Chen:GMEmat} (see also Remark \ref{rmrk:Design_GMEMat_OverallConvex}),
letting $\bm{h} \in \mathbb{R}^{d}$ and $\bm{H} \in \mathbb{R}^{d \times (n-1)}$ such that
$\begin{bmatrix}\bm{h} &\bm{H}\end{bmatrix} \coloneqq  \bm{A}\bm{S}_{\text{\rm 1d}}$,
we can show that
\begin{align*}
\bm{B}^{\top}\bm{B} = \frac{\theta}{\lambda}\bm{H}^{\top}(\bm{I} - \bm{h}\bm{h}^{\dagger})\bm{H} \in \mathbb{R}^{(n-1)\times(n-1)}
\end{align*}
with any $\theta \in [0, 1]$
satisfies \eqref{eq:OverallConvexCondition},
where $\bm{h}^{\dagger} = \bm{h}^{\top}/\|\bm{h}\|^2$ if $\bm{h} \neq \bm{0}$
and $\bm{h}^{\dagger} = \bm{0}^{\top}$ if $\bm{h} = \bm{0}$.

We compare the proposed GME-TGV penalty with
the TGV penalty given in Example \ref{exmp:MI_penalties}(b)
to investigate the effectiveness of enhancement by the proposed framework.
For reference, we also show
the GME-TV penalty $(\|\cdot\|_1)_{\bm{B}}\circ\bm{D}_{\text{\rm 1d}}$ presented in \cite{Abe:LiGME}
and the TV penalty $\|\cdot\|_1\circ\bm{D}_{\text{\rm 1d}}$.
All the penalties are combined with the square error $(1/2)\|\bm{y}-\bm{A}\bm{x}\|^2$.

The GME-TGV model is solved by Algorithm \ref{alg:Optimization_GME_MI}
with the representation in Example \ref{exmp:representation_MIpenalty_forOptimization}(b),
the GME-TV model is solved by the iterative algorithm given in \cite{Abe:LiGME}
with an extension \cite{Kitahara:LiGME_MRI} to handle the constraint
$[-1, 1]^n$,
and the TGV and TV models are solved as the special cases with $\bm{B} = \bm{O}$
of the GME-TGV and GME-TV models, respectively.
We terminate the iteration in the same way as in the previous section.
For the GME-TGV model, the computational complexity of Algorithm \ref{alg:Optimization_GME_MI} per iteration
is $\mathcal{O}(dn)$ for matrix-vector multiplication,
$\mathcal{O}(n)$ for the projection $P_{[-1, 1]^n}$,
and $\mathcal{O}(n)$ for the proximity operators (see \ref{appendix:ComputeProxOp}).
This complexity is identical to that of the algorithm for GME-TV.

In Fig.~\ref{fig:NMSE_PiecewiseLinear}, we compare the models in terms of NMSE,
where the results are averaged over $100$ independent trials, and
the regularization parameter $\lambda$ and other tuning parameters
are adjusted to yield the best estimation accuracy for each pair of models and experimental conditions.
The tuning parameters other than $\lambda$ are as follows:
$\alpha$ for GME-TGV and TGV, and
$\theta$ for the overall convexity conditions of GME-TGV and GME-TV.
As is apparent from Fig.~\ref{fig:NMSE_PiecewiseLinear}, the GME-TGV model exhibits the best estimation accuracy,
with a remarkable improvement over the TGV model.

We also discuss examples of the estimates 
for a trial in Fig.~\ref{fig:trial_experiment_PiecewiseLinear}.
In Fig.~\ref{fig:trial_experiment_PiecewiseLinear}(a) for the TV model,
we observe the so-called stair-casing effect.
The GME-TV model does not resolve the stair-casing effect in the slanted segments in Fig.~\ref{fig:trial_experiment_PiecewiseLinear}(b),
while recovering signal discontinuity more effectively than the TV model.
The TGV model depicted in Fig.~\ref{fig:trial_experiment_PiecewiseLinear}(c)
mostly does not exhibit the stair-casing effect, but underestimates the discontinuous
jumps of the signal and its first derivative.
In Fig.~\ref{fig:trial_experiment_PiecewiseLinear}(d),
we observe that the GME-TGV model significantly mitigates the underestimation of the jumps
of the signal and its first derivative, while being free from the stair-casing effect.

\section{Conclusion}
\label{sect:Conclusion}
This paper presented a convex-nonconvex framework to mitigate the underestimation tendencies of MI penalties. We designed the GME-MI penalty, where the generalized Moreau envelope of the MI penalty is subtracted from it.
While the GME-MI penalty is nonconvex in general, we derived an overall convexity condition for the GME-MI regularized least-squares model. Moreover, we presented a proximal splitting algorithm guaranteed to converge to a globally optimal solution of the GME-MI model under the overall convexity condition. Numerical examples considered in the enhancement scenarios of the LOP-$\ell_2/\ell_1$ and TGV penalties demonstrate the effectiveness of the proposed framework.

\section*{Acknowledgment}
This work was supported in part by the Japan Society for the Promotion of Science under Grants-in-Aid 21K17827, 21H01592, and 24H00272.
The author would like to thank Isao Yamada, Daichi Kitahara, Yi Zhang, and Wataru Yata
for their helpful comments on the original manuscript.
The author would also like to thank the anonymous reviewers for their valuable comments.

\appendix

\section{Validity of Assumption \ref{Assumption:phi_Coresive_Proper_BoundedBelow}}
\label{proof:validity_phi_CoresiveAndProper}
We prove that Assumption \ref{Assumption:phi_Coresive_Proper_BoundedBelow}
automatically hold for the seed functions presented in Example \ref{exmp:MI_penalties} as follows.

{\it For Example \ref{exmp:MI_penalties}(a):}
From the definition \eqref{eq:def:seed_function_LOPl2l1} with \eqref{eq:def:perspective_quadratic_func}
and \eqref{eq:def:indicator_func_l1_ball}, it is clear that
$\varphi_{\alpha}^{\text{\rm LOP}}$ is bounded below
for any $\alpha \in \mathbb{R}_{+}$.
Since $\bm{D}$ is the discrete difference operator, for any constant vector $\bm{c} = (c,c,\ldots,c) \in \mathbb{R}^{m}$,
we have $\bm{D}\bm{c} = \bm{0}$, and thus $\iota_{B_1^{\alpha}}(\bm{D}\bm{c}) = 0$
holds
for any $\alpha \in \mathbb{R}_{+}$ by the definition \eqref{eq:def:indicator_func_l1_ball}.
Meanwhile, by the definition \eqref{eq:def:perspective_quadratic_func}, $\sum_{i=1}^{m}h(u_i,c)$ takes a finite value for any
$\bm{u} \in \mathbb{R}^{m}$ and $c \in \mathbb{R}_{++}$.
Altogether,
$\varphi_{\alpha}^{\text{\rm LOP}}(\bm{u},\cdot)$ is proper
for every $\bm{u} \in \mathbb{R}^{m}$ and $\alpha \in \mathbb{R}_{+}$.
Since $h$ is coercive by \cite[Lemma 1]{Kuroda:BlockSparse}
and $\iota_{B_1^{\alpha}}$ is bounded below for any $\alpha \in \mathbb{R}_{+}$,
$\varphi_{\alpha}^{\text{\rm LOP}}(\bm{u},\cdot)$ is coercive for every $\bm{u} \in \mathbb{R}^{m}$ and $\alpha \in \mathbb{R}_{+}$.

{\it For Example \ref{exmp:MI_penalties}(b):}
From the definition \eqref{eq:def:seed_function_TGV}, it is clear that $\varphi_{\alpha}^{\text{\rm TGV}}$ is bounded below
for any $\alpha \in (0,1)$.
Due to
$\mathrm{dom}(\|\cdot\|_{2,1}^{\mathcal{G}_{1}}) = \mathbb{R}^{m}$
and $\mathrm{dom}(\|\cdot\|_{2,1}^{\mathcal{G}_{2}}) = \mathbb{R}^{p}$,
$\varphi_{\alpha}^{\text{\rm TGV}}(\bm{u},\cdot)$
is proper for every $\bm{u} \in \mathbb{R}^{m}$ and $\alpha \in (0,1)$.
Since $\|\cdot\|_{2,1}^{\mathcal{G}_{1}}$ is a norm, we have
\begin{align*}
\|\bm{u}-\bm{\sigma}\|_{2,1}^{\mathcal{G}_{1}} \geq \bigl|\|\bm{\sigma}\|_{2,1}^{\mathcal{G}_{1}} - \|\bm{u}\|_{2,1}^{\mathcal{G}_{1}}\bigr|,
\end{align*}
which implies that $\|\bm{u}-\cdot\|_{2,1}^{\mathcal{G}_{1}}$ is coercive for any fixed $\bm{u} \in \mathbb{R}^{m}$.
Thus, since $\|\tilde{\bm{D}}\cdot\|_{2,1}^{\mathcal{G}_{2}}$ is bounded below,
$\varphi_{\alpha}^{\text{\rm TGV}}(\bm{u},\cdot)$ is coercive for every $\bm{u} \in \mathbb{R}^{m}$ and $\alpha \in (0,1)$.

\section{Validity of Assumption \ref{asmp:form_seed_func_for_optimization}}
\label{appendix:ComputeProxOp}
We show that efficient exact methods are available for computing the proximity operators of $f$ and $g$ in Example \ref{exmp:representation_MIpenalty_forOptimization}
used to represent the seed functions in the form of \eqref{eq:represent_seed_function}.

{\it For Example \ref{exmp:representation_MIpenalty_forOptimization}(a):}
Since $f(\bm{u},\bm{\sigma}) = \sum_{i=1}^{m}h(u_i,\sigma_i)$
is the separable sum of $h$ in \eqref{eq:def:perspective_quadratic_func},
based on \cite[Example 2.4]{Combettes:PerspectiveML},
for any $(\bm{u},\bm{\sigma}) \in \mathbb{R}^{m}\times\mathbb{R}^{m}$ and $\gamma \in \mathbb{R}_{++}$,
we have $\mathrm{prox}_{\gamma f}(\bm{u},\bm{\sigma})
= (\mathrm{prox}_{\gamma h}(u_i,\sigma_i))_{i=1}^{m}$ with
\begin{align*}
\mathrm{prox}_{\gamma h}(u_i,\sigma_i)
=\begin{dcases}
(0, 0), &\text{\rm if }2\gamma\sigma_{i}+u_{i}^2 \leq \gamma^2;\\
\left(0, \sigma_i-\frac{\gamma}{2}\right), &\mbox{if } u_{i} = 0 \text{ {\rm and} }\sigma_{i} > \frac{\gamma}{2};\\
\left(u_{i} - \gamma t_{i} \frac{u_{i}}{|u_{i}|},\sigma_{i}+\gamma\frac{t_{i}^2-1}{2}\right), &\text{\rm otherwise},
\end{dcases}
\end{align*}
where $t_{i} \in \mathbb{R}_{++}$ is the unique positive root of 
\begin{align*}
t_{i}^3 + \left(\frac{2}{\gamma}\sigma_{i}+1\right)t_{i}-\frac{2}{\gamma} |u_{i}| = 0,
\end{align*}
which can be solved using Cardano's formula as follows \cite{Kuroda:BlockSparse,Bauschke:RootCubic}.
Let $p_{i} = 2\sigma_{i}/\gamma+1$ and $\Delta_{i} = u_{i}^2/\gamma^2+p_{i}^3/27$. Then,
\begin{align*}
t_{i} = \begin{dcases}
\sqrt[3]{\frac{|u_{i}|}{\gamma} + \sqrt{\Delta_{i}}} + \sqrt[3]{\frac{|u_{i}|}{\gamma} - \sqrt{\Delta_{i}}}, &
\text{{\rm if} } \Delta_{i} \geq 0;\\
2\sqrt{-\frac{p_{i}}{3} } \cos\left(\frac{\arctan(\gamma\sqrt{-\Delta_{i}}/|u_{i}|) }{3} \right), &\text{{\rm if} } \Delta_{i} < 0,
\end{dcases}
\end{align*}
where $\sqrt[3]{\cdot}$ denotes the real cubic root.
Altogether, $\mathrm{prox}_{\gamma f}$ for Example \ref{exmp:representation_MIpenalty_forOptimization}(a) can be computed with
$\mathcal{O}(m)$ operations.

Since $\iota_{B_1^{\alpha}}$ in \eqref{eq:def:indicator_func_l1_ball} is the indicator function of the $\ell_1$
norm ball,
$\mathrm{prox}_{\gamma \iota_{B_1^{\alpha}}}$ with any $\gamma \in \mathbb{R}_{++}$ reduces
to the $\ell_1$ ball projection $P_{B_1^{\alpha}}$.
There exist algorithms for computing $P_{B_1^{\alpha}}$ with
$\mathcal{O}(p)$ expected complexity (see, e.g., \cite{Condat:l1ballprojection}).

{\it For Example \ref{exmp:representation_MIpenalty_forOptimization}(b):}
We can represent $f(\bm{u},\bm{\sigma}) = \alpha\|\bm{u}-\bm{\sigma}\|_{2,1}^{\mathcal{G}_{1}}$
as $f = \alpha\|\cdot \|_{2,1}^{\mathcal{G}_{1}} \circ \bm{U}$
with $\bm{U}\coloneqq \begin{bmatrix}\bm{I} & -\bm{I} \end{bmatrix} \in \mathbb{R}^{m\times 2m}$.
Since $\bm{U}\bm{U}^{\top} = 2\bm{I}$,
based on \cite[Proposition 24.14]{BC:ConvexAnalysis},
for any $(\bm{u},\bm{\sigma}) \in \mathbb{R}^{m}\times\mathbb{R}^{m}$, $\gamma \in \mathbb{R}_{++}$,
and $\alpha \in (0,1)$,
we have
\begin{align*}
\mathrm{prox}_{\gamma f}(\bm{u},\bm{\sigma}) 
= \frac{1}{2}\left(\bm{u}+\bm{\sigma}+\mathrm{prox}_{2 \gamma \alpha\|\cdot \|_{2,1}^{\mathcal{G}_{1}}}(\bm{u}-\bm{\sigma}),\bm{u}+\bm{\sigma}-\mathrm{prox}_{2 \gamma \alpha\|\cdot \|_{2,1}^{\mathcal{G}_{1}}}(\bm{u}-\bm{\sigma})\right).
\end{align*}
Since $\mathcal{G}_{1} = (\mathcal{I}_j)_{j=1}^{r}$ is a partition,
the proximity operator of
$\kappa\|\cdot\|_{2,1}^{\mathcal{G}_{1}}$ with $\kappa = 2 \gamma \alpha$ can be computed for any
$\bm{w} \in \mathbb{R}^{m}$ by
\begin{align*}
[\mathrm{prox}_{\kappa\|\cdot\|_{2,1}^{\mathcal{G}_1}}(\bm{w})]_{\mathcal{I}_j}
= \left(1-\frac{\kappa}{\max\{\kappa, \|\bm{w}_{\mathcal{I}_j}\|_2\}}\right)\bm{w}_{\mathcal{I}_j}
\end{align*}
for each $j = 1,\ldots,r$. The proximity operator of $g = (1-\alpha) \|\cdot\|_{2,1}^{\mathcal{G}_{2}}$
can be computed similarly.
Altogether, $\mathrm{prox}_{\gamma f}$
and $\mathrm{prox}_{\gamma g}$
for Example \ref{exmp:representation_MIpenalty_forOptimization}(b)
 can be computed with
$\mathcal{O}(m)$
and $\mathcal{O}(p)$
operations, respectively.

\section{Validity of Assumption \ref{Assumption:QualificationConditions}}
\label{proof:QualificationConditions}
We prove that Assumption \ref{Assumption:QualificationConditions}
is
automatically satisfied for the seed functions in Example \ref{exmp:MI_penalties} with their representations in Example \ref{exmp:representation_MIpenalty_forOptimization} as follows.

{\it For Example \ref{exmp:MI_penalties}(a) with Example \ref{exmp:representation_MIpenalty_forOptimization}(a):}
By the definition \eqref{eq:def:seed_function_LOPl2l1} with \eqref{eq:def:perspective_quadratic_func}
and \eqref{eq:def:indicator_func_l1_ball},
Assumption \ref{Assumption:QualificationConditions}(i)(a) is clear.
From the definitions \eqref{eq:def:seed_function_LOPl2l1} and \eqref{eq:def:L_bar}, we obtain
\begin{align*}
\mathrm{dom}(\varphi_{\alpha}^{\text{\rm LOP}})- \mathrm{ran}(\bar{\bm{L}})
\supset (\mathbb{R}^{m}\times S_{\alpha}) - 
(\mathrm{ran}(\bm{L})\times\mathbb{R}^{m}),
\end{align*}
where $S_{\alpha} \coloneqq  \mathbb{R}^{m}_{++}\cap\{\bm{\sigma}\in\mathbb{R}^{m}\mid\|\bm{D}\bm{\sigma}\|_1 \leq \alpha\}$.
We have $S_{\alpha} \neq \varnothing$ for any $\alpha \in \mathbb{R}_{+}$ because
$\bm{D}\bm{c} = \bm{0}$ holds for any $\bm{c} = (c,c,\ldots,c) \in \mathbb{R}_{++}^{m}$,
as $\bm{D}$ is the discrete difference operator.
Thus, for any $\alpha \in \mathbb{R}_{+}$,
we have
$\mathrm{dom}(\varphi_{\alpha}^{\text{\rm LOP}})- \mathrm{ran}(\bar{\bm{L}}) = \mathbb{R}^{m}\times\mathbb{R}^{m}$,
implying that Assumption \ref{Assumption:QualificationConditions}(ii) holds.
From the definition \eqref{eq:def:M_bar} and the setting of Example \ref{exmp:representation_MIpenalty_forOptimization}(a), we obtain
\begin{align*}
\mathrm{dom}(g) -\bar{\bm{M}}(\mathrm{dom}(f)) = B_{1}^{\alpha} - \bm{D}(\mathbb{R}_{+}^{m}) = B_{1}^{\alpha} - \mathrm{ran}(\bm{D}),
\end{align*}
where $B_{1}^{\alpha} \coloneqq  \{\bm{\xi}\in\mathbb{R}^{p}\mid\|\bm{\xi}\|_1 \leq \alpha\}$,
and the last equality follows from the relation $\bm{D}(\mathbb{R}_{+}^{m}) = \mathrm{ran}(\bm{D})$. This relation holds because, for any $\bm{\sigma} \in \mathbb{R}^{m}$,
$\bm{D}\bm{\sigma} = \bm{D}(\bm{\sigma}+\bm{c})$
and $\bm{\sigma}+\bm{c} \in \mathbb{R}_{+}^{m}$
can be satisfied
by setting $\bm{c} =  (c,c,\ldots,c) \in \mathbb{R}_{+}^{m}$ with $c \geq \max \{|\sigma_i|\}_{i=1}^{m}$.
For any $\alpha \in \mathbb{R}_{+}$,
since $B_{1}^{\alpha} - \mathrm{ran}(\bm{D})$
is a nonempty symmetric convex set, Assumption \ref{Assumption:QualificationConditions}(iii) holds by \cite[Example 6.10]{BC:ConvexAnalysis}.

{\it For Example \ref{exmp:MI_penalties}(b) with Example \ref{exmp:representation_MIpenalty_forOptimization}(b):}
Clear from the definition \eqref{eq:def:seed_function_TGV},
$\mathrm{dom}(\|\cdot\|_{2,1}^{\mathcal{G}_{1}}) = \mathbb{R}^{m}$, 
$\mathrm{dom}(\|\cdot\|_{2,1}^{\mathcal{G}_{2}}) = \mathbb{R}^{p}$,
and the setting of Example \ref{exmp:representation_MIpenalty_forOptimization}(b).

\section{Proof of Positive Definiteness of \texorpdfstring{$\bm{P}$}{P}}
\label{proof:positive_definite_P}
Using $\bm{\Pi}$ in \eqref{eq:def:permutationOp},
we can represent $\bm{P}$ in \eqref{eq:def:positiveSelfAdjOp} as
\begin{align}
\label{eq:representPositiveSelfAdjOpByPermuteBlockDiag}
\bm{P} = \bm{\Pi}^{\top}\begin{bmatrix}
\bm{P}_1 &  &\\
 & \bm{P}_2 & \\
 &  & \bm{P}_3
\end{bmatrix}\bm{\Pi},
\end{align}
where
\begin{align*}
\bm{P}_{1}&\coloneqq \begin{bmatrix}
(1/\gamma_1)\bm{I} & -\lambda\bm{L}^{\top}\bm{B}^{\top}\bm{B} & -\lambda\bm{L}^{\top}\\
-\lambda \bm{B}^{\top}\bm{B}\bm{L} & (\lambda/\gamma_3)\bm{I} & \bm{O}\\
-\lambda \bm{L} & \bm{O} & \lambda \bm{I}
\end{bmatrix},\\
\bm{P}_{2}&\coloneqq \lambda\begin{bmatrix}
(1/\gamma_2)\bm{I} & -\bm{I} & -\bm{M}^{\top}\\
- \bm{I} & \bm{I} & \bm{O}\\
-\bm{M} & \bm{O} & \bm{I}
\end{bmatrix},\\
\bm{P}_{3}&\coloneqq \lambda\begin{bmatrix}
(1/\gamma_3)\bm{I} &  -\bm{M}^{\top}\\
-\bm{M} & (1/\gamma_4)\bm{I}
\end{bmatrix}.
\end{align*}
Since $\bm{\Pi}$ is the permutation matrix,
we have
\begin{align*}
\bm{P} \succ \bm{O} \Leftrightarrow \bm{P}_1 \succ \bm{O}, \bm{P}_2\succ \bm{O}, \bm{P}_3\succ \bm{O}.
\end{align*}
The property of the Schur complement \cite[Theorem 7.7.7]{HJ:MatrixAnalysis} yields
\begin{align*}
\bm{P}_2 \succ \bm{O} &\Leftrightarrow \frac{1}{\gamma_2}\bm{I} -\bm{I}-\bm{M}^{\top}\bm{M} \succ \bm{O},\\
\bm{P}_3 \succ \bm{O} &\Leftrightarrow \frac{1}{\gamma_4}\bm{I} -\gamma_3\bm{M}\bm{M}^{\top} \succ \bm{O},
\end{align*}
and thus $\bm{P}_2 \succ \bm{O}$ and $\bm{P}_3 \succ \bm{O}$ are
guaranteed by the second and the fourth conditions in \eqref{eq:cond_stepSizeLikeParam}, respectively.
Applying the property of the Schur complement to $\bm{P}_1$, we obtain
\begin{align*}
\bm{P}_1 \succ \bm{O} \Leftrightarrow {}&
\frac{1}{\gamma_1}\bm{I}
- \lambda\gamma_3\bm{L}^{\top}(\bm{B}^{\top}\bm{B})^2\bm{L}
-\lambda\bm{L}^{\top}\bm{L} \succ \bm{O}\\
\Leftrightarrow {}&
\frac{1}{\gamma_1}\bm{I}-\frac{\kappa}{2}\bm{A}^{\top}\bm{A}-\lambda\bm{L}^{\top}\bm{L}+\frac{\kappa}{2}\bm{A}^{\top}\bm{A}-\lambda\gamma_3\bm{L}^{\top}(\bm{B}^{\top}\bm{B})^2\bm{L}\succ \bm{O}.
\end{align*}
Since $(1/\gamma_1)\bm{I}-(\kappa/2)\bm{A}^{\top}\bm{A}-\lambda\bm{L}^{\top}\bm{L} \succ \bm{O}$
is guaranteed by the first condition in \eqref{eq:cond_stepSizeLikeParam},
it suffices to prove $(\kappa/2)\bm{A}^{\top}\bm{A}-\lambda\gamma_3\bm{L}^{\top}(\bm{B}^{\top}\bm{B})^2\bm{L}\succeq \bm{O}$.
Since the third condition in \eqref{eq:cond_stepSizeLikeParam} implies
$\gamma_3\|\bm{B}\|_{\mathrm{op}}^2 \leq  \kappa/2$,
for any $\bm{x} \in \mathbb{R}^{n}$, we obtain
\begin{align*}
\langle\bm{x},\gamma_3\bm{L}^{\top}(\bm{B}^{\top}\bm{B})^2&\bm{L}\bm{x} \rangle
=\gamma_3\|\bm{B}^{\top}\bm{B}\bm{L}\bm{x}\|^2
\leq
\gamma_3\|\bm{B}\|_{\mathrm{op}}^2\|\bm{B}\bm{L}\bm{x}\|^2
\leq \frac{\kappa}{2}\|\bm{B}\bm{L}\bm{x}\|^2.
\end{align*}
Thus, we have
\begin{align*}
(\forall \bm{x} \in \mathbb{R}^{n})\quad\left\langle\bm{x},
\left(\frac{\kappa}{2}\bm{A}^{\top}\bm{A}-\lambda\gamma_3\bm{L}^{\top}(\bm{B}^{\top}\bm{B})^2\bm{L}\right)\bm{x} \right\rangle
\geq \frac{\kappa}{2}\langle\bm{x},(\bm{A}^{\top}\bm{A} - \lambda \bm{L}^{\top}\bm{B}^{\top}\bm{B}\bm{L})\bm{x} \rangle \geq 0,
\end{align*}
where the last inequality follows from the condition \eqref{eq:OverallConvexCondition}.

\section{Proof of Averaged Nonexpansiveness of \texorpdfstring{$\mathrm{Id}-\bm{P}^{-1}\circ H$}{Id-P-1H}}
\label{proof:AveragedNonexpansive_AffineOp_H}
To prove the $1/\kappa$-averaged nonexpansiveness of $\mathrm{Id}-\bm{P}^{-1}\circ H$,
it suffices to prove that $\mathrm{Id}-\kappa\bm{P}^{-1}\circ H$ is nonexpansive, in $(\mathcal{H},\langle\cdot,\cdot\rangle_{\bm{P}},\|\cdot\|_{\bm{P}})$.
Define $\bm{R}\in\mathbb{R}^{(n+2m)\times(n+2m)}$ by
\begin{align*}
\bm{R} \coloneqq  \begin{bmatrix}
\bm{A}^{\top}\bm{A} - \lambda \bm{L}^{\top}\bm{B}^{\top}\bm{B}\bm{L} &  &\\
 & \lambda \bm{B}^{\top}\bm{B} & \\
 &  & \bm{O}
\end{bmatrix}.
\end{align*}
Then, $\bm{R} \succeq \bm{O}$ holds by the condition \eqref{eq:OverallConvexCondition}.
For $\bm{z} = (\bm{x},
\bm{\sigma},
\bm{v},
\bm{\tau},
\bm{r},
\bm{\eta},
\bm{\xi},
\bm{\zeta}) \in \mathcal{H}$,
denote $\breve{\bm{z}} \coloneqq  (\bm{x},\bm{v},\bm{r}) \in \mathbb{R}^{n+2m}$.
From the definitions of $H$ and $\bm{\Pi}$ in \eqref{eq:def:AffineOp_H} and \eqref{eq:def:permutationOp}, respectively,
and the expression of $\bm{P}$ in \eqref{eq:representPositiveSelfAdjOpByPermuteBlockDiag},
for any $\bm{z}_1,\bm{z}_2 \in \mathcal{H}$, we have
\begin{align*}
H(\bm{z}_1)-H(\bm{z}_2) &= 
\bm{\Pi}^{\top}(\bm{R}(\breve{\bm{z}}_1-\breve{\bm{z}}_2),\bm{0},\bm{0},\bm{0},\bm{0},\bm{0}),\\
\bm{P}^{-1}[H(\bm{z}_1)-H(\bm{z}_2)] &= \bm{\Pi}^{\top}(\bm{P}_1^{-1}\bm{R}(\breve{\bm{z}}_1-\breve{\bm{z}}_2),\bm{0},\bm{0},\bm{0},\bm{0},\bm{0}).
\end{align*}
Using these expressions, for any $\bm{z}_1,\bm{z}_2 \in \mathcal{H}$, we obtain
\begin{align*}
&\|(\mathrm{Id}-\kappa\bm{P}^{-1}\circ H)(\bm{z}_1)-(\mathrm{Id}-\kappa\bm{P}^{-1}\circ H)(\bm{z}_2)\|_{\bm{P}}^2\\
={}&\|(\bm{z}_1-\bm{z}_2)-\kappa\bm{P}^{-1}[H(\bm{z}_1)-H(\bm{z}_2)]\|_{\bm{P}}^2\\
={}&
\|\bm{z}_1-\bm{z}_2 \|_{\bm{P}}^2-2\kappa\langle\bm{z}_1-\bm{z}_2,H(\bm{z}_1)-H(\bm{z}_2)\rangle
+\kappa^2\langle\bm{P}^{-1}[H(\bm{z}_1)-H(\bm{z}_2)],H(\bm{z}_1)-H(\bm{z}_2)\rangle\\
={}&
\|\bm{z}_1-\bm{z}_2 \|_{\bm{P}}^2-2\kappa\langle\breve{\bm{z}}_1-\breve{\bm{z}}_2,\bm{R}(\breve{\bm{z}}_1-\breve{\bm{z}}_2)\rangle +\kappa^2\langle\bm{P}_1^{-1}\bm{R}(\breve{\bm{z}}_1-\breve{\bm{z}}_2),\bm{R}(\breve{\bm{z}}_1-\breve{\bm{z}}_2) \rangle
\\
={}&\|\bm{z}_1-\bm{z}_2 \|_{\bm{P}}^2
-2\kappa\left\langle\breve{\bm{z}}_1-\breve{\bm{z}}_2, \left(\bm{R}-\frac{\kappa}{2}\bm{R}\bm{P}_1^{-1}\bm{R}\right)(\breve{\bm{z}}_1-\breve{\bm{z}}_2) \right\rangle.
\end{align*}
Thus,
$\mathrm{Id}-\kappa\bm{P}^{-1}\circ H$ is nonexpansive if and only if $\bm{R}-(\kappa/2)\bm{R}\bm{P}_1^{-1}\bm{R} \succeq \bm{O}$.
Let $\bm{R}^{\dagger}$ denote the Moore-Penrose pseudo-inverse of $\bm{R}$.
Since $\bm{R}\bm{R}^{\dagger}\bm{R} = \bm{R}$,
by the property of the generalized Schur complement \cite[Theorem 1.20]{HZ:GeneralizedSchurComplement}, we have
\begin{align*}
\bm{R}-\frac{\kappa}{2}\bm{R}\bm{P}_1^{-1}\bm{R} \succeq \bm{O} \Leftrightarrow 
\begin{bmatrix}
\bm{R}& \bm{R}\\
\bm{R}& \frac{2}{\kappa}\bm{P}_1
\end{bmatrix}\succeq \bm{O} \Leftrightarrow
\bm{P}_1 - \frac{\kappa}{2}\bm{R} \succeq \bm{O}.
\end{align*}
We can rewrite $\bm{P}_1 - (\kappa/2)\bm{R}$ as
\begin{align*}
\bm{P}_1 - \frac{\kappa}{2}\bm{R}
= \begin{bmatrix}
\frac{1}{\gamma_1}\bm{I}-\frac{\kappa}{2}\bm{A}^{\top}\bm{A} & \bm{O} & -\lambda\bm{L}^{\top}\\
\bm{O}& \bm{O} & \bm{O}\\
-\lambda \bm{L} & \bm{O} & \lambda \bm{I}
\end{bmatrix}
+\lambda\begin{bmatrix}
\frac{\kappa}{2}\bm{L}^{\top}\bm{B}^{\top}\bm{B}\bm{L} & -\bm{L}^{\top}\bm{B}^{\top}\bm{B} &\bm{O}\\
 -\bm{B}^{\top}\bm{B}\bm{L}& \frac{1}{\gamma_3}\bm{I} - \frac{\kappa}{2}\bm{B}^{\top}\bm{B} &\bm{O} \\
 \bm{O}&\bm{O}  & \bm{O}
\end{bmatrix},
\end{align*}
and thus it suffices to prove that each term on the right-hand side is positive semidefinite. 
The first term is positive semidefinite
by the first condition in \eqref{eq:cond_stepSizeLikeParam}
and the property of the Schur complement.
For the second term, we have
\begin{align*}
\begin{bmatrix}
\frac{\kappa}{2}\bm{L}^{\top}\bm{B}^{\top}\bm{B}\bm{L} & -\bm{L}^{\top}\bm{B}^{\top}\bm{B}\\
 -\bm{B}^{\top}\bm{B}\bm{L}& \frac{1}{\gamma_3}\bm{I} - \frac{\kappa}{2}\bm{B}^{\top}\bm{B}
\end{bmatrix}
=\begin{bmatrix}
\bm{L}& \\
& \bm{I}
\end{bmatrix}^{\top}
\begin{bmatrix}
\frac{\kappa}{2}\bm{B}^{\top}\bm{B} & -\bm{B}^{\top}\bm{B}\\
 -\bm{B}^{\top}\bm{B}& \frac{1}{\gamma_3}\bm{I} - \frac{\kappa}{2}\bm{B}^{\top}\bm{B}
\end{bmatrix}
\begin{bmatrix}
\bm{L}& \\
& \bm{I}
\end{bmatrix},
\end{align*}
and by the property of the generalized Schur complement,
\begin{align*}
\begin{bmatrix}
\frac{\kappa}{2}\bm{B}^{\top}\bm{B} & -\bm{B}^{\top}\bm{B}\\
 -\bm{B}^{\top}\bm{B}& \frac{1}{\gamma_3}\bm{I} - \frac{\kappa}{2}\bm{B}^{\top}\bm{B}
\end{bmatrix} \succeq \bm{O}
\Leftrightarrow
 \frac{1}{\gamma_3}\bm{I} - \frac{\kappa}{2}\bm{B}^{\top}\bm{B} - \frac{2}{\kappa}\bm{B}^{\top}\bm{B} \succeq \bm{O}.
\end{align*}
Thus, the second term is positive semidefinite by the third condition in \eqref{eq:cond_stepSizeLikeParam}.

\bibliographystyle{elsarticle-num} 

\end{document}